\documentclass[a4paper, 10pt, american]{amsart}

\usepackage{times,latexsym,amssymb}
\usepackage{amsmath,amsthm,bm}
\usepackage{xcolor}
\usepackage[colorlinks,pdfpagelabels,pdfstartview = FitH,bookmarksopen
= true,bookmarksnumbered = true,linkcolor = blue,plainpages =
false,hypertexnames = false,citecolor = red,pagebackref=false]{hyperref}

\usepackage{amsbsy}
\usepackage{amstext}
\usepackage{amssymb}
\usepackage{esint}
\usepackage{stmaryrd}
\usepackage[pdftex]{graphicx}
\usepackage{floatflt}
\usepackage{appendix}

\allowdisplaybreaks
\newtheorem{theorem}{Theorem}
\newtheorem{lemma}[theorem]{Lemma}
\newtheorem{definition}[theorem]{Definition}

\newtheorem{remark}[theorem]{Remark}
\numberwithin{theorem}{section}
\numberwithin{equation}{section}

\def\N{\mathbb{N}}
\def\R{\mathbb{R}}

\newcommand\sfI{{\boldsymbol{\mathsf I}}}

\newcommand\sfm{{\boldsymbol{\mathsf m}}}
\newcommand\sfA{{\boldsymbol{\mathsf A}}}
\newcommand\sfE{{\boldsymbol{\mathsf E}}}
\newcommand\sfa{{\boldsymbol{\mathsf a}}}
\newcommand\sff{{\boldsymbol{\mathsf f}}}

\newcommand\sfF{{\boldsymbol{\mathsf F}}}

\renewcommand{\d}{\mathrm{d}}
\newcommand{\dx}{\mathrm{d}x}

\newcommand{\dt}{\mathrm{d}t}

\renewcommand{\epsilon}{\varepsilon}

\DeclareMathOperator{\Div}{div}

\DeclareMathOperator{\loc}{loc}

\renewcommand{\rho}{\varrho}

\def\eqn#1$$#2$${\begin{equation}\label#1#2\end{equation}}

\let\TeXchi\chi
\newbox\chibox
\setbox0 \hbox{\mathsurround0pt $\TeXchi$}
\setbox\chibox \hbox{\raise\dp0 \box 0 }
\def\chi{\copy\chibox}

\def\Xint#1{\mathchoice
    {\XXint\displaystyle\textstyle{#1}}%
    {\XXint\textstyle\scriptstyle{#1}}%
    {\XXint\scriptstyle\scriptscriptstyle{#1}}%
    {\XXint\scriptscriptstyle\scriptscriptstyle{#1}}%
    \!\int}
\def\XXint#1#2#3{\setbox0=\hbox{$#1{#2#3}{\int}$}
    \vcenter{\hbox{$#2#3$}}\kern-0.5\wd0}
\def\bint{\Xint-}
\def\dashint{\Xint{\raise4pt\hbox to7pt{\hrulefill}}}

\def\Xiint#1{\mathchoice
    {\XXiint\displaystyle\textstyle{#1}}%
    {\XXiint\textstyle\scriptstyle{#1}}%
    {\XXiint\scriptstyle\scriptscriptstyle{#1}}%
    {\XXiint\scriptscriptstyle\scriptscriptstyle{#1}}%
    \!\iint}
\def\XXiint#1#2#3{\setbox0=\hbox{$#1{#2#3}{\iint}$}
    \vcenter{\hbox{$#2#3$}}\kern-0.5\wd0}
\def\biint{\Xiint{-\!-}}

 \subjclass[2020]{35K40, 35B65, 35K65, 35K67}

\begin{document}
\title[]{Calder\'on-Zygmund estimates for parabolic systems with $p$-growth and non-divergence data}

\date{\today}

\author[P. Andrade]{P\^edra Andrade}
\address{P\^edra Andrade\\
Fachbereich Mathematik, Universit\"at Salzburg\\
Hellbrunner Str. 34, 5020 Salzburg, Austria}
\email{pedra.andrade@plus.ac.at}

\author[K. Moring]{Kristian Moring}
\address{Kristian Moring\\
Fachbereich Mathematik, Universit\"at Salzburg\\
Hellbrunner Str. 34, 5020 Salzburg, Austria}
\email{kristian.moring@plus.ac.at}

\setcounter{tocdepth}{1}

\maketitle


\begin{abstract}
We obtain Calderón--Zygmund estimates for weak solutions to nonlinear parabolic systems with polynomial $p$-growth and the right-hand side in non-divergence form. 
Our approach does not require differentiability of the vector field with respect to the gradient variable and provides a unified treatment of both the singular and degenerate regimes. In particular, we establish local higher-integrability estimates for the gradient of the solution under appropriate dependency on the integrability of the right-hand side.
\end{abstract}

\section{Introduction}
In this paper, we develop a Calderón--Zygmund theory for weak solutions to nonlinear
second-order parabolic systems of the form 
\begin{equation}\label{par-sys}
    \partial_t u - \Div \sfA (x,t,Du) = \sff \qquad \mbox{ with\,  $p> \frac{2n}{n+2}$},
\end{equation}
in a space-time cylinder $\Omega_T:=\Omega\times [0,T)$, where $\Omega \subset \R^n$ is bounded and open set, $T > 0$, $u \colon \Omega_T \to \R^N$, $\sff \colon \Omega_T \to \R^N$ is in a proper Lebesgue space and $\sfA \colon \Omega_T \times \R^{Nn} \to \R^{Nn}$ is a vector field with polynomial growth rate $p$ specified in the Section~\ref{structural-conditions}. In other words, our goal is to obtain quantitative grandient estimates in terms of the integrability of the datum $\sff$.

In the case of linear elliptic equations, Calderón--Zygmund estimates originate from the theory of singular integrals developed by Calderón and Zygmund \cite{Calderon-Zygmund-1952}.  In the nonlinear elliptic setting, pioneering contributions for $p$-Laplace-type equations
with the right-hand side in divergence form were obtained by Iwaniec~\cite{Iwaniec}. More precisely, Iwaniec considered equations of the form
\[
\Div\bigl(|Du|^{p-2}Du\bigr)
=
\Div\bigl(|\sfF|^{p-2}\sfF\bigr),
\qquad p>1,
\]
and established higher integrability estimates for the gradient in terms of the
datum $\sfF$. In particular, the corresponding Calderón--Zygmund estimate can be expressed as
\begin{equation}\label{eq:C-Z}
\sfF\in L_{\mathrm{loc}}^q(\Omega, \R^{Nn})
\qquad\Longrightarrow\qquad
Du\in L_{\mathrm{loc}}^q(\Omega, \R^{Nn}),
\qquad q\ge p,    
\end{equation}
for $N =1$. Building upon Iwaniec's pioneering work, DiBenedetto and Manfredi~\cite{DiBenedetto-Manfredi}
extended these results to systems of $p$-Laplace-type equations, establishing
the corresponding gradient estimates in the vectorial setting $(N > 1)$. A further step toward more general nonlinear elliptic equations was taken by Caffarelli and Peral~\cite{Caff-Peral}, who introduced a comparison method for deriving Calderón–Zygmund estimates.

While the previous results focused on divergence-form right-hand sides, a Bogovski\u{\i} type construction, see \cite{Bogo, Galdi}, provides a framework for treating non-divergence-form data. By representing the datum $\sff$ as
\[
\sff=\Div\bigl(|\sfF|^{p-2}\sfF\bigr),
\]
which gives
\[
\sff\in L^{\frac{ns}{n(p-1)+s}}_{\mathrm{loc}}(\Omega,\mathbb{R}^{N}) \qquad
\Longrightarrow \qquad
Du\in L^s_{\mathrm{loc}}(\Omega,\mathbb{R}^{Nn}),
\qquad s\geq p,
\]
where the exponent is sharp.

Further extensions of this theory were obtained by Kinnunen and Zhou in \cite{Kinnunen-Zhou-1,Kinnunen-Zhou-2}, who extended the Calderón--Zygmund estimates to elliptic $p$-Laplace type equations with coefficients satisfying a
vanishing mean oscillation (VMO) condition. Moreover, Byun and Wang~\cite{Byun-Wang} established estimates of the form \eqref{eq:C-Z} for elliptic systems under more general $p$-growth assumptions on the vector field. We refer the reader to~\cite{Mingione:Measure-data} for a comprehensive Calderón--Zygmund theory for elliptic problems with measure data.  In the nonlocal setting, Calderón--Zygmund-type estimates for fractional
$p$-Laplace equations were obtained in \cite{BDLM-1,BDLM-2} under the
assumption that the coefficients $\sfa$ are Hölder continuous, for equations of
the form
\[
(-\sfa \Delta_p)^{\sigma}u=\sff
\qquad \text{in }\Omega.
\]
In particular, in~\cite{BDLM-2}, the following sharp result was obtained:
\[
\sff \in L^{\frac{ns}{n(p-1)+s(1-p(1-\sigma))}}_{\mathrm{loc}}(\Omega)
\quad\Longrightarrow\quad
Du\in L^s_{\mathrm{loc}}(\Omega,\mathbb{R}^n),
\qquad s\geq p .
\]
Furthermore, the approach developed in \cite{BDLM-2} for treating the right-hand side in non-divergence form provides the framework underlying the arguments used here. 

The extension of Calderón--Zygmund estimates to the parabolic setting remained an open problem for some time. Significant difficulty in this case stems from the lack of the homogeneity and scaling properties available in the elliptic case when $p\neq2$, which prevents the direct use of maximal-operator techniques and homogeneous a priori estimates. In ~\cite{Acerbi-Min}, Acerbi and Mingione were able to circumvent these difficulties by introducing a new technique to obtain Calderón--Zygmund estimates for parabolic $p$-Laplace-type systems with VMO coefficients of the form 
\[
\partial_t u-\Div\bigl(\sfa(x,t)|Du|^{p-2}Du\bigr)
=
\Div\bigl(|\sfF|^{p-2}\sfF\bigr),
\qquad p>\frac{2n}{n+2}.
\]
In particular, their result provides the parabolic counterpart of the Calderón--Zygmund estimate in \eqref{eq:C-Z}. A key ingredient of their  approach is the use of the so-called \emph{intrinsic parabolic cylinders} from~\cite{Kinnunen-Lewis:1}, see also~\cite{DiBe, DiBenedetto_Holder}. These are cylinders of the form
\[
Q_{\rho}^{(\lambda)}(x_o,t_o)
=
B_{\rho}(x_o)\times
\bigl(t_o-\lambda^{2-p}{\rho}^2,t_o + \lambda^{2-p}{\rho}^2 \bigr),
\]
where the parameter $\lambda$ is adapted to the solution heuristically through the condition 
\[
\biint_{Q_{\rho}^{
(\lambda)}(x_o,t_o)}|Du|^p\,\dx\dt\approx \lambda^p.
\]
Thus, the geometry of the cylinder is determined by the size of the gradient of the solution on the cylinder itself.
This intrinsic geometry allows one to obtain Calderón--Zygmund estimates directly from PDE arguments, without relying on harmonic-analysis techniques.

Duzaar, Mingione, and Steffen~\cite{Du-Mi-St} extended this approach to more general nonlinear systems with $p$-growth in the superquadratic regime, while Scheven \cite{Scheven-1} further developed the theory in the singular subquadratic case. More precisely, they considered systems of the form
\begin{equation}\label{sys:general-p-growth}
\partial_t u-\Div \sfA(x,t,Du)
=
\Div\bigl(|\sfF|^{p-2}\sfF\bigr),
\qquad p>\frac{2n}{n+2}.
\end{equation}
satisfying ellipticity conditions, as well as a continuity condition in $x$ normalized by the natural $p$-growth, and differentiablity assumption with respect the gradient variable. It worth mentioning the authors also treated system as in~\eqref{sys:general-p-growth} with VMO-coefficients with respect to the spatial variable.  The estimates established in these works can be stated as follows: there exists $\varepsilon>0$ such that
\[
\sfF \in L^q_{\mathrm{loc}}(\Omega_T,\R^{Nn})
\quad\Longrightarrow\quad
Du \in L^q_{\mathrm{loc}}(\Omega_T,\R^{Nn}),
\qquad p \le q < p + \frac{4}{n} + \varepsilon.
\]

While a substantial theory is available for divergence-form data, the case of non-divergence-form right-hand sides presents additional difficulties, as the datum no longer possesses the derivative structure that can be directly
exploited in the weak formulation. For parabolic systems with non-divergence-form right-hand sides, optimal Calderón--Zygmund estimates are known in the linear case $p=2$, see \cite{By-Ki-Ku, Du-Mi-pot}. In the nonlinear case, sharp Calderón--Zygmund estimates under VMO assumptions on the
coefficients were obtained in \cite{ABDM}, where the authors considered systems of the form
\[
\partial_t u-\Div\bigl(\sfa(x,t)|Du|^{p-2}Du\bigr)
=\sff,
\qquad p>\frac{2n}{n+2}.
\]
They established the sharp Calderón--Zygmund estimate
\[
\sff\in
L^{\frac{(n+2)s}{n(p-1)+p+s}}_{\mathrm{loc}}
(\Omega_T,\mathbb{R}^{Nn})
\quad\Longrightarrow\quad
Du\in L^s_{\mathrm{loc}}(\Omega_T,\mathbb{R}^{Nn}),
\qquad s>p,
\]
which recovers the optimal exponent in the linear case $p=2$.

Building on the developments in \cite{ABDM,Du-Mi-St,Scheven-1}, we establish Calderón--Zygmund estimates for a broader class of nonlinear parabolic systems with polynomial $p$-growth and non-divergence-form right-hand sides, without requiring differentiability of the coefficients with respect to the gradient variable. Moreover, our approach provides a unified treatment of the singular case $p<2$ and the degenerate case $p\geq2$. Our main result is the following:

\begin{theorem}\label{thm:main}
Let $p>\frac{2n}{n+2}$. Suppose that
\[
u\in C\big((0,T);L^2_{\mathrm{loc}}(\Omega,\mathbb{R}^N)\big)\cap 
L^p_{\mathrm{loc}}\big(0,T;W^{1,p}_{\mathrm{loc}}(\Omega,\mathbb{R}^N)\big)
\]
is a local weak solution to the parabolic system~\eqref{par-sys} in the sense of Definition~\ref{def:weak_solution} when~\eqref{bounds-A}, and either~\eqref{upper-bound-A} or~\eqref{eq:A-ea}-\eqref{eq:e-VMO} are in force. 

Then, there exists $\varepsilon = \varepsilon (n,N,p,\nu,L)>0$ such that
\[
|\sff| \in L^{\frac{(n+2)s}{n(p-1) + p+ s}}_{\mathrm{loc}}(\Omega_T) \implies |Du|\in L^s_{\mathrm{loc}}(\Omega_T),
\]
whenever $s \in (p,  p + \frac{4}{n} + \varepsilon)$.
Moreover, there exists a constant 
$
c = c(n,N,p,\nu,L,s,\omega(\cdot))
$
such that for every parabolic cylinder $Q_R\Subset \Omega_T$  with $R \in (0,1]$, we have
\begin{align*}
    \bigg[\biint_{Q_{\frac12 R}} &|D u|^s  \, \dx\dt\bigg]^{\frac1s} \\
    &\leq 
    c\Bigg[
    \biint_{Q_R}|Du|^p\,\dx\dt
    +
    \bigg(
    \biint_{Q_R}|R\sff|^{\frac{(n+2)s}{n(p-1)+ p + s }}\,\dx\dt + 1
    \bigg)^{\frac{p}{s} + \frac{p}{n(p-1)+p}} \Bigg]^\frac{d}{p}, 
\end{align*}
where the
parabolic scaling deficit is defined by
\begin{align}\label{eq:deficit}
    1
    \le 
    d
    :=
    \left\{
        \begin{array}{cl}
            \frac{p}{2}, & \mbox{if $p\ge 2$,} \\[5pt]
            \frac{2p}{p(n+2)-2n}, & \mbox{if $\frac{2n}{n+2}<p<2$.}
        \end{array} 
    \right.
\end{align}
\end{theorem}

\begin{remark}
As a consequence of the higher-integrability theory developed in \cite{Du-Mi-St,Scheven-1}, our approach yields local $L^s$-estimates for $Du$ in the range
\[
p\leq s<p+\frac{4}{n}+\varepsilon.
\]
It is worth mentioning that the necessity of restricting the range of the Calderón--Zygmund exponent is also reflected already in the elliptic case in the counterexamples of \v{S}verák and Yan~\cite{Sverak-Yan}, who constructed non-Lipschitz minimizers of smooth uniformly convex functionals. Their example indicates that the result without an upper bound for $s$ can be only expected under stronger structural assumptions on the vector field $\sfA$.
\end{remark}

Notice that the quantity $d$ appearing in Theorem~\ref{thm:main}, referred to as the \emph{parabolic scaling deficit}, naturally reflects the non-homogeneous space-time scaling of parabolic problems with $p$-growth when $p\neq 2$. In Theorem~\ref{thm:main}, we assume that the integrability exponent satisfies $s>p$ and that
\[
|\sff|\in
L^{\Theta(s)}_{\mathrm{loc}}(\Omega_T),
\qquad
\Theta(s):=
\frac{(n+2)s}{n(p-1)+p+s},
\]
which yields
\[
|Du|\in L^s_{\mathrm{loc}}(\Omega_T).
\]
Observe that $\Theta$ is increasing in $s$. Moreover, for
$p<s<p^\sharp$, we have
\[
\lim_{s\to p^\sharp}\Theta(s)
=
\frac{(n+2)p^\sharp}
{n(p-1)+p+p^\sharp},
\]
while
\[
\lim_{s\to p}\Theta(s)
=
\frac{(n+2)p}
{n(p-1)+2p}.
\]

Unlike in Theorem~\ref{thm:main}, we may set $q:=\Theta(s) $ and formulate the resulting gain of integrability directly in terms of $q$. This yields the following equivalent formulation of Theorem~\ref{thm:main}.

\begin{theorem}\label{thm:equiv}
    Let $p>\frac{2n}{n+2}$. Assume that
\[
u\in C\big((0,T);L^2_{\mathrm{loc}}(\Omega,\mathbb{R}^N)\big)\cap 
L^p_{\mathrm{loc}}\big(0,T;W^{1,p}_{\mathrm{loc}}(\Omega,\mathbb{R}^N)\big)
\]
is a local weak solution to the parabolic system~\eqref{par-sys} in the sense of Definition~\ref{def:weak_solution} when~\eqref{bounds-A}, and either~\eqref{upper-bound-A} or~\eqref{eq:A-ea}-\eqref{eq:e-VMO} are in force.  Then, there exists $p^\sharp > p+ \frac{4}{n}$ depending only on $n,N,p,\nu$ and $L$, such that
    $$
|\sff| \in L^q_{\loc}(\Omega_T) \implies |Du| \in L^\frac{q(n(p-1) +p)}{n+2-q}_{\loc}(\Omega_T)
$$
whenever $\frac{(n+2)p}{n(p-1) + 2p}< q < \frac{(n+2) p^\sharp}{n(p-1) + p + p^\sharp}$. Moreover, there exists a constant $c = c(n,N,p,\nu,L,q, \omega(\cdot))$ such that for every parabolic cylinder $Q_R \Subset \Omega_T$ with $R \in (0,1]$, we have
\begin{align*}
    \bigg[\biint_{Q_{\frac12 R}} &|D u|^\frac{q(n(p-1)+p)}{n+2-q}  \, \dx\dt\bigg]^{\frac{n+2-q}{q(n(p-1)+p)}} \\
    &\leq 
    c\Bigg[
    \biint_{Q_R}|Du|^p\,\dx\dt
    +
    \bigg(
    \biint_{Q_R}|R\sff|^{q}\,\dx\dt + 1
    \bigg)^{\frac1q \cdot\frac{p(n+2)}{n(p-1)+p}}\Bigg]^\frac{d}{p}.
\end{align*}
\end{theorem}

In Theorem~\ref{thm:equiv}, we suppose that
$|\sff|\in L^q_{\mathrm{loc}}(\Omega_T)$, which implies that
$|Du|\in L^s_{\mathrm{loc}}(\Omega_T)$, where
\[
s:=\frac{q[n(p-1)+p]}{n+2-q}.
\]
Note that the map $q\mapsto \frac{q[n(p-1)+p]}{n+2-q}$ is increasing on $(0,n+2)$ and 
\[
\lim_{q\to\frac{(n+2)p}{n(p-1)+2p}}
\frac{q[n(p-1)+p]}{n+2-q}
=p,
\]
while
\[
\lim_{q\to
\frac{(n+2)p^\sharp}{n(p-1)+p+p^\sharp}}
\frac{q[n(p-1)+p]}{n+2-q}
=p^\sharp.
\]
Hence $p< s < p^{\sharp}$, then we recover the Theorem~\ref{thm:main}.

 The proofs are based on comparison arguments. Solutions of the problem~\eqref{par-sys} are compared with solutions of the homogeneous problem with frozen coefficients in which the vector field is independent of the spatial variable.  The gradient of solutions to such problems, actually, when the vector field is differentiable in space, are known by the results in~\cite{Du-Mi-St,Scheven-1,Scheven-2}, to be integrable to a power strictly greater than $p + \frac{4}{n}$. This higher gradient integrability property is then transferred to the solution of~\eqref{par-sys}, whenever the right-hand side $\sff$ allows. 

For vector fields satisfying weak type continuity assumption~\eqref{upper-bound-A} in space, we compare the solution of~\eqref{par-sys} directly with the frozen homogeneous problem obtained by fixing the first variable in $\sfA(x,t,\xi) $. We also consider the case when $\sfA$ is of separable form as in~\eqref{eq:A-ea}, with VMO coefficients in space. In this case, we exploit comparison arguments in two subsequent steps: first, by comparing the original solution to the solution of the homogeneous problem, and subsequently, the latter to the solution of the frozen problem. In order to deal with coefficients with VMO-type behavior, we exploit the general higher integrability result from~\cite{Kinnunen-Lewis:1}.

\section{Preliminaries}

\subsection{Notation}

We denote points in space-time by $z_o=(x_o,t_o)\in\mathbb{R}^n\times\mathbb{R}$. 
For $\rho>0$, we write $B_\rho(x_o)\subset\mathbb{R}^n$ for the open ball centered at $x_o$ with radius $\rho$, and simply $B_\rho$ if $x_o$ is clear from the context. For $\theta > 0$, we denote the cylinder  centered at $z_o=(x_o,t_o)$ by
\[
Q_{\rho,\theta} (z_o) = B_\rho(x_o)\times
\bigl(t_o-\theta,\;t_o+\theta \bigr),
\]
and for $\lambda > 0$, we introduce the intrinsic parabolic cylinder
\[
Q_\rho^{(\lambda)}(z_o)
:= Q_{\rho,\lambda^{2-p}\rho^2} (z_o).
\]
When $\lambda=1$, the superscript is omitted, and we recover the standard parabolic cylinders $Q_\rho = Q_{\rho,\rho^2}$. Moreover, whenever the center $z_o$ is clear from the context, we omit $z_o$ from the notation $Q_{\rho,\theta} (z_o)$.

\subsection{Notion of weak solutions}
In what follows, we introduce the notion of weak solutions, which will be used throughout the paper.

\begin{definition}\label{def:weak_solution}\upshape
Assume that $\sfA:\Omega_T \times \R^{Nn} \to\R^{Nn}$ satisfies \eqref{bounds-A}, and that $\sff\in L^1_{\mathrm{loc}}(\Omega_T,\R^N)$. 
A function
\[
u\in C\big((0,T); L^2_{\mathrm{loc}}(\Omega,\R^N)\big)\cap
L^p_{\mathrm{loc}}\big(0,T;W^{1,p}_{\mathrm{loc}}(\Omega,\R^N)\big)
\]
is called a local weak solution to \eqref{par-sys} if
\begin{align}\label{weak-solution}
	\iint_{\Omega_T}\big[- u\cdot \partial_t \varphi + \sfA(x,t, Du)\cdot D\varphi\big]\dx\dt
    =
   \iint_{\Omega_T}\sff\cdot\varphi\,\dx\dt
\end{align}
holds for every test function $\varphi\in C_0^\infty(\Omega_T,\R^N)$.
\hfill$\Box$
\end{definition}

\subsection{Structural conditions on the vector field $\sfA$}\label{structural-conditions}
In this section, we formulate the structural assumptions imposed on the nonlinear vector field $\sfA$. 
\begin{equation} \label{bounds-A}
\left\{
\begin{aligned}
&|\sfA(x,t,w)|
\le
L(\mu^2+|w|^2)^{\frac{p-1}{2}},
\\[4pt]
&  (\sfA(x,t,w) - \sfA(x,t, \widetilde w)) \cdot (w- \widetilde w)
\ge
\nu(\mu^2+|w|^2 + |\widetilde w|^2)^{\frac{p-2}{2}}
|w-\widetilde w|^2,
\\[4pt]
&|\sfA (x,t,w) - \sfA(x,t, \widetilde w)| \leq L (\mu^2 + |w|^2 + |\widetilde w|^2)^\frac{p-2}{2} |w - \widetilde w| \quad \text{ if }\,  p \geq 2,
\end{aligned}
\right.
\end{equation}
for any $x, x_o \in \Omega$, $t \in (0,T)$, and $w , \widetilde w \in \R^{Nn}$, where $p > \frac{2n}{n+2}$, $0 < \nu \leq 1 \leq L < \infty$ and $0 \leq \mu \leq 1$. The bound~\eqref{bounds-A}$_3$ is assumed to obtain the higher integrability result in the superquadratic case, Lemma~\ref{lem:HI}, which applies~\cite{Du-Mi-St}.

Furthermore, we need a weak continuity type assumption in the spatial variable. We consider two alternative conditions. First, we suppose that there exists a concave modulus of continuity $\omega : [0,\infty) \to [0,2) $ satisfying
$$
\lim_{R \to 0} \omega (R ) = 0 = \omega (0),
$$
such that
\begin{equation} \label{upper-bound-A}
|\sfA(x,t,w)-\sfA(x_o,t,w)|
\le
L\omega(|x-x_o|)
(\mu^2+|w|^2)^{\frac{p-1}{2}}.
\end{equation}
for every $x, x_o \in \Omega$, $t\in (0, T)$ and $w \in \mathbb{R}^{Nn}$. Alternatively, we consider the case
\begin{equation} \label{eq:A-ea}
\sfA (x,t,\xi) = e(x) \sfa(t,\xi), \quad \text{ where $\sfa$ satisfies~\eqref{bounds-A}},
\end{equation}
and
\begin{equation} \label{eq:e-uniform-ell}
0<\nu \leq e(x) \leq L < \infty \quad \text{ for a.e. } x \in \Omega,
\end{equation}
and  $e \in \mbox{VMO} (\Omega)$, that is, 
\begin{equation} \label{eq:e-VMO}
\lim_{R\to 0}\omega(R) = 0, \quad 
\end{equation}
 where 
 \[
 \, \omega(R) := \sup_{\substack{B_\rho \Subset \Omega\\ 0<\rho\le R}} \bint_{B_\rho} |e(x) - (e)_{B_\rho}| \, \d x.
 \]
\subsection{ Auxiliary results}

For the sake of completeness, we state standard algebraic estimates, see, for instance,  \cite{Acerbi-Min-p(x),BDM-obst} and \cite{BDLM-1}.
\begin{lemma} \label{lem:tech}
Let $k \in \N$, $p > 1$ and $\mu \in [0,1]$. Then, there exists a constant $c = c(k,p)$ such 
$$
(\mu^2 + |a|^2)^\frac{p}{2} \leq c (\mu^2 + |b|^2)^\frac{p}{2} + c (\mu^2 +|a|^2 + |b|^2)^\frac{p-2}{2} |b-a|^2 
$$
for every $a,b \in \R^k$.
\end{lemma}

\begin{lemma} \label{lem:elementary-superlevel}
Let $k \in \N$, $\alpha \in [1,\infty)$, $a,b \in \R^k$ and $|a| \geq K$ for some $K > 0$. Then 
for every $\beta \in [\alpha,\infty)$
we have
\begin{equation*}
    |a|^\alpha \leq 2^\alpha |a-b|^\alpha + 2^{\beta -1} K^{\alpha -\beta} |b|^\beta.
\end{equation*}
\end{lemma}

In what follows, we recall a higher integrability result that was established by Duzaar, Mingione, and Steffen \cite[Lemma 5.9]{Du-Mi-St} in the superquadratic case under a differentiability assumption with respect to the gradient variable.  A careful inspection of the proof reveals that our structural conditions~\eqref{bounds-A}, which are implications of the ones in \cite{Du-Mi-St}, are sufficient for this result. The corresponding  result  in the subquadratic case  is due to Scheven \cite[Lemma 2.5]{Scheven-1} (see also~\cite{Scheven-2}). In addition, the referred higher integrability results in \cite[Lemma 5.9]{Du-Mi-St} and ~\cite[Lemma 2.5]{Scheven-1} apply for vector fields that are differentiable with respect to the spatial variable. For our purposes it is sufficient to apply it for vector fields independent of the spatial variable.

\begin{lemma}\label{lem:HI}
Let $p>\frac{2n}{n+2}$ and ${\sf K }\ge1$. Let 
\[
v\in C\big((0,T);L_{\mathrm{loc}}^2(\Omega,\R^N)\big)
\cap L_{\mathrm{loc}}^p\big(0,T;W_{\mathrm{loc}}^{1,p}(\Omega,\R^N)\big)
 \]
 be a weak solution to \eqref{par-sys}, where $\sfA$ is independent of the spatial variable, and satisfies the structural assumptions~\eqref{bounds-A} with $\sff=0$ in $\Omega_T$. Then there exists $\varepsilon=\varepsilon(n,N,p,L,\nu)>0$ such that for any 
\begin{equation*}
    p \le \gamma < p + \tfrac{4}{n} + \varepsilon,
\end{equation*}
there exists a constant
$c=c(n,N,p, L,\nu,{\sf K })\ge1$ with the following property. If $Q_\rho^{(\lambda)}(z_o)\Subset \Omega_T$ is a parabolic cylinder such that 
\[
\bigg[
\biint_{Q_\rho^{(\lambda)}(z_o)}
|Dv|^p\,\dx\dt\bigg]^\frac1p
\le {\sf K } \lambda \qquad 
\]
holds for some $\lambda \ge 1$, then
\[
\bigg[
\biint_{ Q^{(\lambda)}_{\rho / 2}(z_o)}
|Dv|^{\gamma}\,\dx\dt\bigg]^\frac{1}{\gamma}
\le
c \lambda.
\]
\end{lemma}

\begin{remark}
It is worth emphasizing that one of the novelties of our work is the weakening of the assumptions imposed in \cite{Du-Mi-St}, where differentiability with respect to the gradient variable is required. In our setting, we only assume the structural conditions \eqref{bounds-A}, which is a direct consequence of those structural conditions in \cite{Du-Mi-St}.
\end{remark}

For the treatment of VMO coefficients, we exploit another higher integrability result in a more general setting, that can be recovered from~\cite{Kinnunen-Lewis:1}.

\begin{lemma}\label{lem:HI-KL}
Let $p>\frac{2n}{n+2}$, ${\sf K} \ge1$, and suppose that 
\[
v\in C\big((0,T);L_{\mathrm{loc}}^2(\Omega,\R^N)\big)
\cap L_{\mathrm{loc}}^p\big(0,T;W_{\mathrm{loc}}^{1,p}(\Omega,\R^N)\big)
\]
is a weak solution to \eqref{par-sys} with $\sfA$ satisfying~\eqref{bounds-A} and $\sff=0$ in $\Omega_T$.
Then, there exist constants $\sigma=\sigma(n,N,p,\nu,L)>p$ and 
$c=c(n,N,p,\nu,L,{\sf K})\ge1$ with the following property: if $Q_\rho^{(\lambda)}(z_o) \Subset \Omega_T$ is a parabolic cylinder such that
\[
\bigg[
\biint_{Q_\rho^{(\lambda)}(z_o)}
|Dv|^p\,\dx\dt\bigg]^\frac1p
\le {\sf K} \lambda
\]
holds for some $\lambda \ge 1$, then
\[
\bigg[
\biint_{Q^{(\lambda)}_{\rho/2}(z_o)}
|Dv|^{\sigma}\,\dx\dt\bigg]^\frac{1}{\sigma}
\le
c \lambda.
\]
\end{lemma}

\begin{proof}
    Without loss of generality we assume $z_o=0$. We consider the auxiliary function
 \begin{equation*}
    \tilde v(x,t)
    :=
    \frac{v(\rho x,\lambda^{2-p}\rho^2 t)}{\rho\lambda},
    \qquad
    \widetilde \sfA(x,t, \xi)
    :=
    \lambda^{1-p}\sfA(\rho x,\lambda^{2-p}\rho^2 t, \lambda \xi),
\end{equation*}
with $(x,t)\in Q_1$. Hence, 
$\tilde v\in C(-1,1;L^2(B_1,\R^N))\cap L^p(-1,1;W^{1,p}(B_1,\R^N))$ is a weak solution to 
\begin{equation*}
    \partial_t\tilde v -
    \Div \big( \widetilde\sfA (x,t, D\tilde{v})\big)
    =
    0
    \quad\mbox{in\quad $Q_1$.}
\end{equation*}
By applying~\cite{Kinnunen-Lewis:1}, we have that there exists $\sigma >p$ such that 
$$
\left[ \biint_{Q_\frac12} |D \tilde v|^\sigma\, \d x \d t \right]^\frac{1}{\sigma} \leq c \left[ \left( \biint_{Q_{1}} |D \tilde v|^p \, \d x \d t \right)^\frac{d}{p}  + 1\right],
$$
where $d$ is as in  \eqref{eq:deficit}. Scaling back to $v$ yields
\begin{align*}
    \left[ \biint_{Q_{\frac12 \rho}^{(\lambda)}} |D v|^\sigma\, \d x \d t \right]^\frac{1}{\sigma} &\leq c \lambda \left[ \lambda^{-d} \left( \biint_{Q_\rho^{(\lambda)}} |Dv|^p \, \d x \d t \right)^\frac{d}{p} + 1\right] \leq c  \big[ {\sf K }^d + 1 \big] \lambda.
\end{align*}
\end{proof}

\subsection{Steklov averages}\label{Steklov-averages}

In general, weak solutions do not possess a time derivative in the Sobolev sense. Since it is often useful to employ the weak solution itself as a test function, doing so requires additional justification. For this purpose, we introduce the Steklov averages, defined as follows. For $0<h<T$ and $f \in L^1 (\Omega_T,\R^k)$ with $k \in \N$. We define the \emph{Steklov average} of $f$ by
\[
f_h(x,t)=
\begin{cases}
\displaystyle \frac{1}{h}\int_t^{t+h} f(x,s)\,\mathrm ds,
& t\in(0,T-h),\\[6pt]
0,
& \text{otherwise}.
\end{cases}
\]
Observe that
\begin{equation} \label{eq:time-derivative-stek}
\partial_t f_h (x,t) := \frac{1}{h} [f(x,t+h) - f(x,t)].
\end{equation}
for almost every $(x, t) \in \Omega \times (0, T -h)$. Furthermore, if $f \in L^q(\Omega_T,\R^k)$ for $1\leq q <\infty$, then  $f_h \in L^q(\Omega_T,\R^k)$ and
$$
f_h \xrightarrow{h\to 0} f\quad \text{ in } L^q(\Omega_{T-\delta},\R^k) \ \text{ for any $\delta \in (0,T)$}.
$$
We refer to \cite{DiBe} (see also \cite[Lemma 2.4]{BDGLS})  for further details on the properties of Steklov averages in the study of parabolic equations.

\section{Comparison estimates}
We assume that 
\[
u\in C\big((0,T); L^2_{\mathrm{loc}}(\Omega,\R^N)\big)\cap
L^p_{\mathrm{loc}}\big(0,T;W^{1,p}_{\mathrm{loc}}(\Omega,\R^N)\big)
\]
is a local weak solution to~\eqref{par-sys}. We denote by 
\begin{equation*}
    v\in C\big([t_o-\theta, t_o+\theta),
    L^2(B_\rho (x_o),\R^N)\big)
    \cap
    L^p\big(t_o-\theta, t_o+\theta; W^{1,p}
    (B_\rho (x_o),\R^N)\big)
\end{equation*}
the unique solution to the Cauchy-Dirichlet problem
\begin{equation} \label{CD-homo}
\left\{
\begin{array}{cl}
      \partial_t v - \Div \sfA (x_o,t,Dv)
      = 0,
      & 
      \mbox{in $ Q_{\rho,\theta}(z_o)$,} \\[7pt]
      v = u, & \mbox{on  $\partial_{\rm par} Q_{\rho,\theta}(z_o)$,}
   \end{array}
\right.
\end{equation}
where $\sfA$ satisfies~\eqref{bounds-A} and~\eqref{upper-bound-A} and we define
\begin{equation} \label{eq:q-q'}
    q=p\frac{n+2}{n} \qquad\text{and}\qquad q'=\frac{p(n+2)}{p(n+2)-n}. 
\end{equation}

\begin{lemma}\label{lem:comparison}
Let $p>\frac{2n}{n+2}$, and let $u$ be a weak solution to~\eqref{par-sys} in the sense of Definition~\ref{def:weak_solution}, with $\sfA$ satisfying~\eqref{bounds-A} and~\eqref{upper-bound-A}. Let $Q_{\rho,\theta}\equiv Q_{\rho,\theta}(z_o)\Subset\Omega_T$, and let $v$ denote the unique weak solution to the homogeneous Cauchy--Dirichlet problem~\eqref{CD-homo} in $Q_{\rho,\theta}$. Then, there exists a constant $c = c(n,N,p,\nu,L)$ such that for every $\delta\in(0,1]$ we have
\begin{align*}
&\sup_{\tau\in (t_o-\theta, t_o+ \theta)}
\int_{B_\rho \times\{\tau\}}
|u-v|^2\,\dx
+
\iint_{Q_{\rho,\theta}}
|Du-Dv|^p\,\dx\dt\\
&\quad\le c \left( \delta + \frac{\omega(\rho)^\frac{p}{p-1}}{\delta^\frac{(2-p)_+}{p-1}} \right) \iint_{Q_{\rho,\theta}} (\mu^2 + |Du|^2)^\frac{p}{2} \, \d x \d t \\
&\qquad + \frac{c}{\delta^\frac{(2-p)_+(n+2)}{p(n+2) -n-p}}\left[ \iint_{Q_{\rho,\theta}} |\sff|^{q'} \, \d x \d t \right]^{\frac{1}{q'} \frac{p(n+2) }{p(n+2) -n-p}},
\end{align*}
where $q'$ is defined in~\eqref{eq:q-q'}.
\end{lemma}

\begin{proof}
Since $u$ is a solution to~\eqref{par-sys}  and $v$ is a solution to~\eqref{CD-homo}, we have
\begin{align*}
    &\partial_t (u-v)- \Div \left( \sfA (x_o,t, Du) - \sfA (x_o,t,Dv) \right) \\
    &\quad = \Div \left( \sfA (x,t, Du) - \sfA (x_o,t,Du)\right) + \sff \quad  \mbox{in $Q_{\rho,\theta}$},
\end{align*} 
with $u-v=0$ on $\partial_{\rm par} Q_{\rho,\theta}$.

To handle the time-derivative term, we employ the Steklov average formulation for parabolic systems. More precisely,  we choose a spatial test function $\varphi \in C_0^\infty(B_\rho)$ and a time test function
$\zeta \in C_0^\infty(t_o-\theta,t_o+\theta-h)$ and take
\[
\psi(x, t):= \varphi(x) \frac{1}{h} \int_{t-h}^{t} \zeta(s) \, \d s.
\]
We then use~\eqref{eq:time-derivative-stek}, together with Fubini's theorem, to estimate the contribution of the time-derivative term. Next, we use the definition of the Steklov averages, together with Fubini's theorem, to estimate the remaining terms. Hence, we obtain
\begin{align*}
&\int_{B_\rho \times \{t\}} \partial_t (u_h - v_h)\cdot \varphi + [\sfA(x_o,t,Du) - \sfA(x_o,t,Dv)]_h \cdot D \varphi \, \d x \\
&\quad = \int_{B_\rho \times \{t\}} [\sfA(x,t,Du) - \sfA(x_o,t,Du)]_h \cdot D \varphi + \sff_h \cdot \varphi\, \d x 
\end{align*}
for every $\varphi \in W^{1,p}_0(B_\rho,\R^N)$, $0<h<2 \theta$ and $t \in (t_o - \theta ,t_o + \theta -h)$. Here, the subscript $h$ indicates the Steklov averages defined in the Section~\eqref{Steklov-averages}. Next, we choose $\varphi(x) = \bar u_h (x,t) - \bar v_h (x,t)$, where $\bar w_h = \tfrac12 [w(x,t+h) + w(x,t)]$, $w \in \{u,v\}$. Let $\tau \in (t_o-\theta,t_o+ \theta-h)$, and denote $I_\tau = (t_o-\theta, \tau)$. Integrating over $I_\tau$ yields
\begin{align} \label{eq:steklov-form}
& \iint_{I_\tau \times B_\rho} \partial_t (u_h - v_h)\cdot \varphi + [\sfA(x_o,t,Du) - \sfA(x_o,t,Dv)]_h \cdot D \varphi \, \d x   \d t\nonumber \\
&\quad = \iint_{I_\tau \times B_\rho} [\sfA(x,t,Du) - \sfA(x_o,t,Du)]_h \cdot D \varphi + \sff_h \cdot \varphi\, \d x   \d t.
\end{align} 
Using~\eqref{eq:time-derivative-stek} we obtain
\begin{align*}
&\iint_{I_\tau \times B_\rho} \partial_t (u_h-v_h) (\bar u_h - \bar v_h) \, \d x   \d t\\
&\quad= \frac{1}{2h} \int_{t_o-\theta}^\tau \int_{ B_\rho} |u-v|^2 (x,t+h) - |u-v|^2(x,t) \, \d x  \d t \\
&\quad= \frac{1}{2h}\int_{B_\rho} \left[ \int_{t_o-\theta +h}^{\tau+h} |u-v|^2 (x,t)\, \d t - \int_{t_o-\theta}^{\tau} |u-v|^2 (x,t)\, \d t \right]\, \d x\\
&\quad = \frac{1}{2} \bint_\tau^{\tau+h}\int_{B_\rho} |u-v|^2 (x,t) \, \d x  \d t - \frac12 \bint_{t_o-\theta}^{t_o-\theta+h}  \int_{B_\rho} |u-v|^2 (x,t) \, \d x  \d t\\
&\quad \xrightarrow{h \to 0} \frac12 \int_{{B_\rho} \times \{\tau\}}|u-v|^2(x,t) \, \d x,
\end{align*}
by Lebesgue differentiation theorem and the initial condition. By passing to the limit as $h\to0$ in~\eqref{eq:steklov-form} and using the convergence stated in the display above, we obtain
\begin{align*}
&\frac12 \int_{B_\rho \times \{\tau\}} |u-v|^2 \, \d x  + \iint_{I_\tau \times B_\rho} [\sfA(x_o,t,Du) - \sfA(x_o,t,Dv)] \cdot D(u-v) \, \d x  \d t\\
&\quad = \iint_{ I_\tau \times B_\rho} [\sfA(x,t,Du) - \sfA(x_o,t,Du)] \cdot D (u-v) + \sff \cdot (u-v)\, \d x \d t .\\
&\quad \leq \iint_{Q_{\rho,\theta}} |\sfA(x,t,Du) - \sfA(x_o,t,Du)| |D u-Dv| + |\sff| |u-v| \, \d x \d t
\end{align*}
Since $\tau\in(t_o-\theta,t_o+\theta)$ is arbitrary, we consider the terms on the left hand side separately, and take the supremum with respect to $\tau$ in the first term and let $\tau\uparrow t_o+\theta$ in the second term. We thus obtain
\begin{align*}
&\sup_{\tau \in (t_0 - \theta,  t_0 +\theta)}\int_{B_\rho \times \{\tau\}} |u-v|^2 \, \d x  + \iint_{Q_{\rho,\theta}} [\sfA(x_o,t,Du) - \sfA(x_o,t,Dv)] \cdot D(u-v) \, \d x  \d t\\
&\quad \leq 3 \iint_{Q_{\rho,\theta}} |\sfA(x,t,Du) - \sfA(x_o,t,Du)| |D u-Dv| + |\sff| |u-v| \, \d x \d t \\
&\quad =: 3 \sfI \sfI.
\end{align*}

Now, we consider the case $p \geq 2$. By using~\eqref{bounds-A}$_2$, we have
\begin{align*}
&\iint_{Q_{\rho ,\theta}}  \left( \sfA (x_o,t, Du) - \sfA (x_o,t,Dv)\right) \cdot D(u-v)\, \d x \d t \\
&\quad \geq  \nu \iint_{Q_{\rho ,\theta}} (\mu^2 + |Du|^2 + |D v|^2)^\frac{p-2}{2} |Du - D v|^2\, \d x \d t  \\
&\quad \overset{p \geq 2}{\geq} C \iint_{Q_{\rho ,\theta}}  |Du - D v|^p \, \d x \d t, 
\end{align*}
by applying Lemma~\ref{lem:tech} to obtain the last inequality above and we define
\[
\sfI_2:= \iint_{Q_{\rho ,\theta}} (\mu^2 + |Du|^2 + |D v|^2)^\frac{p-2}{2} |Du - D v|^2\, \d x \d t.
\]
By using~\eqref{upper-bound-A} and Hölder's inequality, we estimate the following:
\begin{align*}
\sfI \sfI &\leq L \omega(\rho) \iint_{Q_{\rho ,\theta}}  (\mu^2 + |Du|^2 )^\frac{p-1}{2} |Du - Dv| \, \d x \d t + \iint_{Q_{\rho ,\theta}}  |\sff ||u-v|\, \d x \d t \\
&\leq \varepsilon \iint_{Q_{\rho ,\theta}}  |Du -Dv|^p \, \d x \d t + \frac{L^\frac{p}{p-1}}{\varepsilon^{\frac{1}{p-1}}} \omega(\rho)^\frac{p}{p-1} \iint_{Q_{\rho ,\theta}}  (\mu^2 + |Du|^2 )^\frac{p}{2} \, \d x \d t \\
&\quad + \left[ \iint_{Q_{\rho ,\theta}}  |\sff|^{q'} \,\d x \d t \right]^\frac{1}{q'} \left[ \iint_{Q_{\rho ,\theta}}  |u-v|^q \, \d x \d t\right]^\frac{1}{q}.
\end{align*}
For the last term, we employ the Gagliardo--Nirenberg inequality \cite[Chapter I, Proposition 3.1]{DiBe}, which yields
\begin{align*}
&\left[ \iint_{Q_{\rho ,\theta}}  |\sff|^{q'} \,\d x \d t \right]^\frac{1}{q'} \left[ \iint_{Q_{\rho ,\theta}}  |u-v|^q \, \d x \d t\right]^\frac{1}{q} \\
&\, \leq c \left[ \iint_{Q_{\rho ,\theta}}  |\sff|^{q'} \,\d x \d t \right]^\frac{1}{q'} \left[ \iint_{Q_{\rho ,\theta}} |Du-Dv|^p \, \d x \d t\right]^\frac{n}{p(n+2)} \sfI_1^\frac{1}{n+2},
\end{align*}
where $c = c(n,p)$ and
\begin{equation}\label{Def-I_1}
 \sfI_1 := \sup_{\tau\in (t_o-\theta, t_o+\theta)}\int_{B_\rho} |u-v|^2 \, \d x.   
\end{equation}
Using Young's inequality with exponents $\frac{p(n+2)}{p(n+2)-n-p}$,  $q$, and $n+2$, we obtain
\begin{align*}
&c \left[ \iint_{Q_{\rho ,\theta}}  |\sff|^{q'} \,\d x \d t \right]^\frac{1}{q'} \left[ \iint_{Q_{\rho ,\theta}} |Du-Dv|^p \, \d x \d t\right]^\frac{n}{p(n+2)} \sfI_1^\frac{1}{n+2} \\
&\le
{\varepsilon}
\left[
\sfI_1
+
\iint_{Q_{\rho,\theta}}
|Du-Dv|^p\,\dx\dt
\right]
+
\frac{c}{{\varepsilon}^{\frac{n+p}{p(n+2)-n-p}}}
\left[
\iint_{Q_{\rho,\theta}}
|\sff|^{q'}
\,\dx\dt
\right]^{\frac{p(n+2)-n}{p(n+2)-n-p}} .
\end{align*}
 By setting $\tilde c:= \min\left \{ 1 , C\right \}$, 
we can write
\begin{align*}
 \tilde c\left [ \sfI_1 + \iint_{Q_{\rho ,\theta}}  |Du - D v|^p \, \d x \d t\right]
 &\le 6\varepsilon  \left[ \sfI_1 + \iint_{Q_{\rho ,\theta}}  |Du - D v|^p \, \d x \d t\right] \\
 &+ \frac{c}{\varepsilon^{\frac{1}{p-1}}} \omega(\rho)^\frac{p}{p-1} \iint_{Q_{\rho ,\theta}}  (\mu^2 + |Du|^2 )^\frac{p}{2} \, \d x \d t  \\
 &+ \frac{c}{{\varepsilon}^{\frac{n+p}{p(n+2)-n-p}}}
\left[
\iint_{Q_{\rho,\theta}}
|\sff|^{q'}
\,\dx\dt
\right]^{\frac{p(n+2)-n}{p(n+2)-n-p}} .
\end{align*}
Choosing $ \varepsilon = \frac{\tilde c}{12}$, we may absorb the first term on the right-hand side into the left-hand side and we obtain the comparison estimate
\begin{align*}
&\sfI_1 + \iint_{Q_{\rho ,\theta}}  |Du - Dv|^p \, \d x \d t \\
&\quad\leq c \omega(\rho)^\frac{p}{p-1} \iint_{Q_{\rho ,\theta}}  (\mu^2 + |Du|^2)^\frac{p}{2} \, \d x \d t + c \left[ \iint_{Q_{\rho ,\theta}}  |\sff|^{q'} \, \d x \d t \right]^{\frac{1}{q'} \frac{p(n+2)}{p(n+2) - n - p}}
\end{align*}
when $p \geq 2$. In the case $p<2$, we proceed as follows:  for $\delta>0$, we write
\begin{align*}
|Du-Dv|^p
&=
\delta^{\frac{2-p}{2}}
\big(\mu^2 + |Du|^2+|Dv|^2\big)^{\frac{(2-p)p}{4}}
\delta^{-\frac{2-p}{2}}
\big(\mu^2 + |Du|^2+|Dv|^2\big)^{\frac{(p-2)p}{4}}
|Du-Dv|^p .
\end{align*}
An application of Young's inequality with exponents $\frac{2}{2-p}$ and $\frac{2}{p}$ implies
$$
|Du - Dv|^p \leq \delta (\mu^2 + |Du|^2 + |Dv|^2)^\frac{p}{2} + \delta^{-\frac{2-p}{p}} (\mu^2 + |Du|^2 + |Dv|^2)^\frac{p-2}{2} |Du - Dv|^2
$$
and furthermore
$$
(\mu^2 + |Du|^2 + |Dv|^2)^\frac{p}{2} \leq 3^\frac{p}{2} (\mu^2 + |Du|^2)^\frac{p}{2} + 2^\frac{p}{2} |Du - Dv|^p.
$$
By restricting $\delta \leq \frac{1}{4}$, we have
\begin{equation} \label{eq:p-difference-sing}
|Du-Dv|^p \leq 6 \delta (\mu^2 + |Du|^2)^\frac{p}{2} + 2 \delta^{-\frac{2-p}{p}} (\mu^2 + |Du|^2 + |Dv|^2)^\frac{p-2}{2} |Du - Dv|^2.
\end{equation}
Recall that
\begin{align*}
\sfI \sfI &\leq \varepsilon \left[ \sfI_1 + \iint_{Q_{\rho ,\theta}}  |Du-Dv|^p \, \d x \d t \right] + c \frac{\omega(\rho)^{p'}}{\varepsilon^\frac{1}{p-1}} \iint_{Q_{\rho ,\theta}}  (\mu^2 + |Du|^p)^\frac{p}{2} \, \d x \d t \\
&\quad + \frac{c}{\varepsilon^\frac{n+p}{p(n+2) - n - p}} \left[ \iint_{Q_{\rho ,\theta}}  |\sff|^{q'} \, \d x \d t \right]^{\frac{1}{q'} \frac{p(n+2)}{p(n+2) -n-p}}
\end{align*}
Applying the estimates above, we obtain
\begin{align*}
&\sfI_1 + \iint_{Q_{\rho ,\theta}}  |Du - Dv|^p\, \d x \d t \\
&\quad \leq 6 \delta \iint_{Q_{\rho ,\theta}}  (\mu^2 + |Du|^2)^\frac{p}{2} \, \d x \d t + \sfI_1+ 2 \delta^{-\frac{2-p}{p}} \sfI_2 \\
&\quad \leq  6 \delta \iint_{Q_{\rho ,\theta}}  (\mu^2 + |Du|^2)^\frac{p}{2} \, \d x \d t + 2 \sfI \sfI + \frac{c}{\delta^\frac{2-p}{p}} \sfI \sfI \\
&\quad \leq 6 \delta \iint_{Q_{\rho ,\theta}}  (\mu^2 + |Du|^2)^\frac{p}{2} \, \d x \d t +  C \frac{\varepsilon}{\delta^\frac{2-p}{p}} \left[ \sfI_1 + \iint_{Q_{\rho ,\theta}}  |Du-Dv|^p \, \d x \d t \right] \\
&\qquad + \frac{c \omega(\rho)^{p'}}{\delta^\frac{2-p}{p} \varepsilon^\frac{1}{p-1}} \iint_{Q_{\rho ,\theta}}  (\mu^2 + |Du|^2)^\frac{p}{2} \, \d x \d t \\
&\qquad + \frac{c}{\delta^\frac{2-p}{p}\varepsilon^\frac{n+p}{p(n+2) - n - p}} \left[ \iint_{Q_{\rho ,\theta}}  |\sff|^{q'} \, \d x \d t \right]^{\frac{1}{q'} \frac{p(n+2)}{p(n+2) -n-p}}.
\end{align*}
By choosing $\varepsilon = \frac{1}{2C} \delta^\frac{2-p}{p}$, and replacing $12 \delta$ by $\delta$, we conclude
\begin{align*}
&\sfI_1 + \iint_{Q_{\rho ,\theta}} |Du-Dv|^p \, \d x \d t\\
&\quad\leq \delta  \iint_{Q_{\rho ,\theta}} (\mu^2 + |Du|^2)^\frac{p}{2} \, \d x \d t + c \frac{\omega(\rho)^{p'}}{\delta^\frac{2-p}{p-1}} \iint_{Q_{\rho ,\theta}} (\mu^2 + |Du|^2)^\frac{p}{2} \, \d x \d t \\
&\qquad + \frac{c}{\delta^\frac{(2-p) (n+2)}{p(n+2) -n-p}}\left[ \iint_{Q_{\rho ,\theta}} |\sff|^{q'} \, \d x \d t \right]^{\frac{1}{q'} \frac{p(n+2)}{p(n+2) -n-p}},
\end{align*}
which holds for any $\delta \in (0,1]$, and completes the proof in the case $p < 2$.

\end{proof}

\subsection{Comparison estimates for parabolic systems with VMO coefficients}
In the following, we turn our attention to the system with VMO coefficients. Let 
$$
v \in C \big( [t_o-\theta,t_o+\theta) ; L^2(B_\rho(x_o),\R^N)) \cap L^p(t_o-\theta,t_o+\theta ; W^{1,p}(B_\rho(x_o),\R^N) \big)
$$
be the unique weak solution to
\begin{equation} \label{CD-homo-2}
\left\{
\begin{array}{cl}
      \partial_t v - \Div \sfA (x,t,Dv)
      = 0,
      & 
      \mbox{in $ Q_{\rho,\theta}(z_o)$,} \\[7pt]
      v = u, & \mbox{on  $\partial_{\rm par} Q_{\rho,\theta}(z_o)$,}
   \end{array}
\right.
\end{equation}
where $\sfA$ is of the form~\eqref{eq:A-ea} under assumptions~\eqref{eq:e-uniform-ell} and~\eqref{eq:e-VMO}.
\begin{lemma} \label{lem:comparison-vmo}
Let $p>\frac{2n}{n+2}$, and let $u$ is a  weak solution to~\eqref{par-sys} 
in the sense of Definition~\ref{def:weak_solution}, with $\sfA$ satisfying ~\eqref{eq:A-ea}-\eqref{eq:e-VMO}. Let  $Q_{\rho,\theta} \equiv Q_{\rho,\theta}(z_o) \Subset \Omega_T$ and let $v$ denote the unique weak solution to the homogeneous Cauchy-Dirichlet problem~\eqref{CD-homo-2} on $Q_{\rho,\theta}$. Then, there exists a constant  $c = c(n,N,p,\nu,L)$ such that for every $\delta\in(0,1]$ we have
\begin{align*}
&\sup_{\tau\in (t_o-\theta, t_o+\theta)}
\int_{B_\rho \times\{\tau\}}
|u-v|^2\,\dx
+
\iint_{Q_{\rho,\theta}}
|Du-Dv|^p\,\dx\dt\\
&\quad\le \delta  \iint_{Q_{\rho,\theta}} (\mu^2 + |Du|^2)^\frac{p}{2} \, \d x \d t  + \frac{c}{\delta^\frac{(2-p)_+(n+2)}{p(n+2) -n-p}}\left[ \iint_{Q_{\rho,\theta}} |\sff|^{q'} \, \d x \d t \right]^{\frac{1}{q'} \frac{p(n+2) }{p(n+2) -n-p}},
\end{align*}    
where $q'$ is defined in~\eqref{eq:q-q'}.
\end{lemma}
\begin{proof}
Observe that
\begin{align*}
\partial_t (u-v) - \Div (e(x) (\sfa(t,Du) - \sfa(t,Dv))) = \sff \quad \text{ in } Q_{\rho,\theta},
\end{align*}
with $u=v$ on $\partial_{\mathrm{par}} Q_{\rho,\theta}$. Arguing as before, using the Steklov average formulation, we obtain
\begin{align*}
\sup_{\tau \in (t_0 - \theta,  t_0 +\theta)}\int_{B_\rho \times \{\tau\}} |u-v|^2 \, \d x
&+
\iint_{Q_{\rho,\theta}}
e(x)
\big(
\sfa(t,Du) - \sfa(t,Dv)
\big)
\cdot
D(u-v)
\, \dx\dt\\
&\le
3 \iint_{Q_{\rho,\theta}}|\sff| |u -v|\,\dx\dt.
\end{align*}
Since $e$ satisfies~\eqref{eq:e-uniform-ell} and $\sfa$ satisfies~\eqref{bounds-A}$_2$, we have
\begin{align*}
&\sfI_1 + c \iint_{Q_{\rho,\theta}} (\mu^2 + |Du|^2 + |Dv|^2)^\frac{p-2}{2} |Du- Dv|^2\, \d x \d t \\
&\quad \leq 3 \iint_{Q_{\rho,\theta}} |\sff| |u-v| \, \d x \d t, 
\end{align*}
where $\sfI_1$ is defined as in ~\eqref{Def-I_1}. Next, we distinguish between the superquadratic and subquadratic regimes. In the case $p \geq 2$, we treat the term on the right hand side exactly as in the proof of Lemma~\ref{lem:comparison}, yielding
$$
\sfI_1 + \iint_{Q_{\rho,\theta}} |Du- Dv|^p \, \d x\d t \leq c \left[ \iint_{Q_{\rho,\theta}} |\sff|^{q'} \, \d x \d t \right]^{\frac{1}{q'} \frac{p(n+2)}{p(n+2) - n - p}}.
$$
In the case $ 1<p<2$, by similar considerations as in the proof of Lemma~\ref{lem:comparison}, we have
\begin{align*}
\sfI_1 +  \iint_{Q_{\rho,\theta}} |Du- Dv|^p \, \d x\d t &\leq \delta \iint_{Q_{\rho,\theta}} (\mu^2 + |Du|^2)^\frac{p}{2} \, \d x \d t \\
&\quad + \frac{c}{\delta^\frac{(2-p)(n+2)}{p(n+2) -n-p}}\left[ \iint_{Q_{\rho,\theta}} |\sff|^{q'} \, \d x \d t \right]^{\frac{1}{q'} \frac{p(n+2) }{p(n+2) -n-p}}.
\end{align*}
\end{proof}

\section{Proof of Theorem~\ref{thm:main}}
\subsection{Stopping time argument}
Let $R_o \leq 1$ to be determined later, depending only on $n,N,p,\nu,L, s$ and $\omega(\cdot)$, and consider $R \leq R_o$. In this part, we consider the parabolic cylinder $Q_R (\tilde z_o) = B_R (\tilde x_o) \times (\tilde t_o - R^2 , \tilde t_o + R^2) \Subset \Omega_T$. Without loss of generality, we suppose that $\tilde z_o = (\tilde x_o, \tilde t_o)$ coincides with the origin and omit it from the notation. We consider nested cylinders $Q_{\tfrac12 R} \subset Q_{r_1} \subset Q_{r_2} \subset Q_R$, where $\tfrac12 R \leq r_1 < r_2 \leq R$. We set
\begin{align}\label{def:lambda_o}
\lambda_o^\frac{p}{d} := 1 &+ \biint_{Q_R} |Du|^p \, \d x \d t \nonumber\\
&+ \mathsf{M}^p |Q_R|^\frac{p}{p(n+2) - (n+p)} \left[ \biint_{Q_R} |\sff|^{q'} \, \d x \d t \right]^{\frac{1}{q'} \frac{p(n+2)}{p(n+2) - (n+p)} },
\end{align}
where $d$ is defined in~\eqref{eq:deficit}, $q$ and $q'$ are defined as in~\eqref{eq:q-q'}, and ${\sf M} \geq 1$ will be determined later in terms of  $n,N,p,\nu,L$ and $s$. Observe that the exponent of $|Q_R|$ is related to the exponent of $\biint_{Q_R} |\sff|^{q'} \, \d x \d t$ as follows:
\begin{equation} \label{exponent_average}
 \frac{p}{  p(n + 2) - (n +p)} = \frac{1}{q'}\frac{p(n+2)}{p(n +2) - (n +p)} - 1.   
\end{equation}
Let $\lambda>1$ and define 
\[
\sfm = \max \big\{ 1 , \lambda^\frac{2-p}{2} \big\}.
\]
Notice that for every point $z_o\in Q_{r_1}$, and any $r$ satisfying
\begin{equation}\label{choice-r}
\frac{r_2-r_1}{2^{6}\sfm}
\le r \le
\frac{r_2-r_1}{2 \sfm},
\end{equation}
yields the inclusion
\[
Q_r^{(\lambda)}(z_o) \subset Q_{r_2},
\]
 since
$$
\lambda^{2-p} r^2 \leq \lambda^{2-p} \left( \frac{r_2-r_1}{2\sfm} \right)^2 < r_2^2 - r_1^2.
$$
Next, we set 
\begin{equation}\label{def:B}
    {\sf B}:=
    \bigg(
    \frac{2^{6}R}{r_2-r_1}
    \bigg)^{(n+2) \frac{d}{p}},
\end{equation}
and let
\begin{equation} \label{eq:lambda-geq-Blambda0}
    \lambda > {\sf B} \lambda_o.
\end{equation}
By the same reasoning as in \cite{ABDM}, we can enlarge the domain of integration from 
$Q_r^{(\lambda)}(z_o)$ to $Q_R$ and using 
\eqref{eq:lambda-geq-Blambda0}, we get for every $r$ as in \eqref{choice-r} that 
\begin{align}\label{eq:lambda-sub}\nonumber
    \biint_{Q_r^{(\lambda)}(z_o)}& |Du|^p\,\dx\dt
    +
    {\sf M}^p \big|Q_r^{(\lambda)}\big|^{\frac{p}{p(n+2)-(n+p)}}
    \bigg[
    \biint_{Q_r^{(\lambda)}(z_o)}|\sff|^{q'}
    \,\dx\dt
    \bigg]^{\frac{1}{q'}\frac{p(n+2)}{p(n+2)-(n+p)}}\\\nonumber
    &\le 
    \frac{|Q_R|}{|Q_r^{(\lambda)}|}\Bigg[
    \biint_{Q_R}|Du|^p\,\dx\dt\\\nonumber
    &\qquad\qquad\quad
    +
    {\sf M}^p|Q_R|^{\frac{p}{p(n+2)-(n+p)}}
    \bigg[
    \biint_{Q_R}|\sff|^{q'}\,\dx\dt
    \bigg]^{\frac{1}{q'}\frac{p(n+2)}{p(n+2)-(n+p)}}\Bigg]\\\nonumber
    &<  
    \Big(\frac{R}{r}\Big)^{n+2}\lambda^{p-2}\lambda_o^{\frac{p}{d}} \\
    &\leq \lambda^p.
    \end{align}
     In the final step, we used \eqref{eq:deficit}, \eqref{choice-r}, \eqref{def:B}, and \eqref{eq:lambda-geq-Blambda0}. Next, we consider $\lambda$ satisfying ~\eqref{eq:lambda-geq-Blambda0} and we set the \emph{superlevel set}
 \begin{equation}\label{def-superlevel set}
   \sfE(\lambda, r_1):= \{z_o \in Q_1\, :\, \text{is a Lebesgue point of}\, |Du| \, \text{and}\,  |Du|(z_o)> \lambda\}.   
 \end{equation}
As a consequence of the Lebesgue differentiation theorem, for any $z_o \in \sfE(\lambda, r_1)$, we have
 \[
 \lim_{\rho\rightarrow 0} \biint_{Q_{\rho}^{(\lambda)(z_o)}} |Du |^p  \d x \d t= |Du(z_o)|^p >\lambda^p.
 \]
It follows that, for all sufficiently small $\rho$,
\begin{align*}
    \biint_{Q_\rho^{(\lambda)}(z_o)} &|D u|^p \,\dx\dt\\
    &+
    {\sf M}^p
    \big|Q_\rho^{(\lambda)}\big|^{\frac{p}{p(n+2)-(n+p)}}
    \bigg[
        \biint_{Q_\rho^{(\lambda)}(z_o)} |\sff|^{q'} \,\dx\dt
    \bigg]^{\frac{1}{q'}\frac{p(n+2)}{p(n+2)-(n+p)}}
    >  \lambda^p.
\end{align*}
 By combining the inequality above with~\eqref{eq:lambda-sub} and using the absolute continuity of the integral, there exists a maximal radius
\[
0<\rho_{z_o}<\frac{r_2-r_1}{2^6\sfm}
\]
such that
\begin{align*}
    \biint_{Q_\rho^{(\lambda)}(z_o)} &|D u|^p \,\dx\dt\\
    &+
    {\sf M}^p
    \big|Q_\rho^{(\lambda)}\big|^{\frac{p}{p(n+2)-(n+p)}}
    \bigg[
        \biint_{Q_\rho^{(\lambda)}(z_o)} |\sff|^{q'} \,\dx\dt
    \bigg]^{\frac{1}{q'}\frac{p(n+2)}{p(n+2)-(n+p)}}
    =
    \lambda^p
\end{align*}
and 
\begin{align*}
    \biint_{Q_\rho^{(\lambda)}(z_o)} &|D u|^p \,\dx\dt\\
    &+
    {\sf M}^p
    \big|Q_\rho^{(\lambda)}\big|^{\frac{p}{p(n+2)-(n+p)}}
    \bigg[
        \biint_{Q_\rho^{(\lambda)}(z_o)} |\sff|^{q'} \,\dx\dt
    \bigg]^{\frac{1}{q'}\frac{p(n+2)}{p(n+2)-(n+p)}}
    < \lambda^p,
\end{align*}
for any $\rho_{z_o}<\rho \le \frac{r_2- r_1}{2\sfm}$.

\subsection{A Covering argument for superlevel sets}
 Let $\mathcal{F}$ be a family of intrinsic cylinders consisting of
\[
\left\{
Q_{\rho_{z_o}}^{(\lambda)}(z_o)
: z_o\in\sfE(\lambda,r_1)
\right\},
\]
where $\lambda$ satisfies~\eqref{eq:lambda-geq-Blambda0} and $\rho_{z_o}\in\left(0,\frac{r_2-r_1}{2^6\sfm}\right)$ is determined by the preceding stopping-time argument. By Vitali's covering theorem, there exists a countable pairwise disjoint subfamily
\[
\left\{
Q_i\equiv Q_{\rho_i}^{(\lambda)}(z_i)
\right\}_{i\in\mathbb{N}}
\subset\mathcal{F}
\]
such that
\[
\sfE(\lambda,r_1)
\subset
\bigcup_{i\in\mathbb{N}}
Q_{5\rho_i}^{(\lambda)}(z_i).
\]
Moreover, for every $i\in\mathbb{N}$, the stopping-time construction yields
\begin{align}\label{eq:intr}
    \biint_{Q_i} |D u|^p \,\dx\dt
    + 
    {\sf M}^p|Q_i|^{\frac{p}{p(n+2)-(n+p)}} \bigg[\biint_{Q_i} |\sff|^{q'} \,\dx \dt\bigg]^{\frac{1}{q'}\frac{p(n+2)}{p(n+2)-(n+p)}}  
    &=
    \lambda^p, 
\end{align}
and
\begin{align} \label{eq:stopping-time-above-maxradius}\nonumber
    \biint_{Q_{\rho}^{(\lambda)}(z_i)}& |D u|^p \,\dx\dt \\
    &+ 
    {\sf M}^p\big|Q_{\rho}^{(\lambda)}\big|^{\frac{p}{p(n+2)-(n+p)}} \bigg[\biint_{Q_{\rho}^{(\lambda)}(z_i)} |\sff|^{q'} \,\dx \dt\bigg]^{\frac{1}{q'}\frac{p(n+2)}{p(n+2)-(n+p)}}  
    <
    \lambda^p
\end{align}
for every $\rho\in \big(\rho_{i}, \tfrac{ r_2-r_1}{2 \sfm}\big]$. To simplify the notation, we define
\[
    2^\ell Q_i:= Q_{2^\ell\rho_i}^{(\lambda)}(z_i),
    \qquad i,\ell\in\mathbb N.
\]
By construction of $\rho_i$, we have \(2^5Q_i \subset Q_{r_2} \subset Q_R\) for every \(i\in\mathbb N\).

\subsection{Proof of Theorem~\ref{thm:main} under the condition~\eqref{upper-bound-A}} \label{sec:proof-cont-A}

We begin by applying a comparison argument with solutions to the corresponding homogeneous system, where the spatial variable of the vector field is frozen. Let 
\begin{equation} \label{eq:vi-space}
\begin{split}
    v_i&\in C
    \big(
    [t_i-\lambda^{2-p}(2^5\rho_i)^2,
    t_i+\lambda^{2-p}(2^5\rho_i)^2);
    L^2 \big(B_{2^5\rho_i}(x_i);\mathbb{R}^N
    \big)\big)\\
    &\phantom{\in\,}
    \cap
    L^p\big(
    t_i-\lambda^{2-p}(2^5\rho_i)^2,
    t_i+\lambda^{2-p}(2^5\rho_i)^2;
    W^{1,p} \big(B_{2^5\rho_i}(x_i);\mathbb{R}^N
    \big)\big)
\end{split}
\end{equation}
be a weak solution to
\begin{equation*}
\left\{
\begin{array}{cl}
      \partial_t v_i - \Div \sfA (x_i,t,Dv_i)
      = 0,
      & 
      \mbox{in $ 2^5 Q_i$,} \\[7pt]
      v_i = u, & \mbox{on  $\partial_{\rm par} \big(2^5 Q_i\big)$,}
   \end{array}
\right.
\end{equation*}
where $u$ is a weak solution to~\eqref{par-sys}. An application of Lemma~\ref{lem:comparison} for $u - v_i$ yields
\begin{align} \label{eq:comparison1}
&\iint_{2^5Q_i} |Du-Dv_i|^p \, \d x \d t \nonumber \\
&\quad\leq c\left( \delta + \frac{\omega(R)^{p'}}{\delta^\frac{(2-p)_+}{p-1}}\right) \iint_{2^5Q_i} (\mu^2 + |Du|^2)^\frac{p}{2} \, \d x \d t \\
&\qquad + \frac{c}{\delta^\frac{(2-p)_+ (n+2)}{p(n+2) -n-p}}\left[ \iint_{2^5Q_i} |\sff|^{q'} \, \d x \d t \right]^{\frac{1}{q'} \frac{p(n+2)}{p(n+2) -n-p}}, \nonumber
\end{align}
with $c = c (n,N,p,\nu,L)$.
By choosing $\delta = 1$ in~\eqref{eq:comparison1} and  recalling $\mu \leq 1 \leq \lambda$, $\omega \leq 2$  and  let $\sf M> 1$, we have
\begin{align} \label{eq:v_i-subintinsic}
\biint_{2^5 Q_i} &|D v_i|^p \, \d x \d t \nonumber  \\
&\leq c \biint_{2^5Q_i} |Du -D v_i|^p \, \d x \d t + c \biint_{2^5 Q_i} |Du|^p \, \d x \d t  \nonumber  \\
& \le  c \biint_{2^5Q_i} (\lambda^2 + |Du|^2)^\frac{p}{2} \, \d x \d t + c \biint_{2^5 Q_i} |Du|^p \, \d x \d t \\
& \qquad\qquad+ |2^5Q_i|^{\frac{p}{p(n +2) - (n +p)}} \left[ \biint_{2^5Q_i} |\sff|^{q'} \, \d x \d t \right]^{\frac{1}{q'} \frac{p(n+2)}{p(n+2) -n-p}}\nonumber  \\
& \le c \Bigg[ \lambda^p + \biint_{2^5 Q_i} |Du|^p \, \d x \d t \nonumber\\
&\qquad\qquad +  {\sf M}^p|2^5Q_i|^{\frac{p}{p(n +2) - (n +p)}} \left[ \biint_{2^5Q_i} |\sff|^{q'} \, \d x \d t \right]^{\frac{1}{q'} \frac{p(n+2)}{p(n+2) -n-p}}\Bigg]\nonumber \\
&\leq {\sf c} \lambda^p \nonumber, 
\end{align}
with ${\sf c}  = {\sf c} (n,N,p,\nu,L)$.  In other words, the cylinder $2^5 Q_i$ is subintrinsic with respect to $v_i$.

\subsubsection{Estimates on the superlevel sets of $|Du|$} Let $\varepsilon = \varepsilon(n,N,p,L,\nu)> 0$ be the one from Lemma~\ref{lem:HI}. Fix $s \in (p, p+\tfrac{4}{n} + \varepsilon)$, and let $\gamma = \gamma(s,p,n,N,L,\nu) = \tfrac12 (s+ p+\tfrac{4}{n} + \varepsilon )$. 
Applying Lemma~\ref{lem:elementary-superlevel} with $(k, a, b, \alpha, \beta, K)$ replaced by $(Nn, Du, Dv_i, p, \gamma, K\lambda)$ and integrating over $2^3Q_i\cap\{|Du|>K\lambda\}$, we obtain
\begin{align} \label{eq:level-set-first}
&\iint_{2^3 Q_i \cap \{|Du|> K \lambda\}} |Du|^p \, \d x \d t \nonumber\\
&\quad \leq 2^p \iint_{2^5 Q_i } |Du - D v_i|^p \, \d x \d t + \frac{2^{\gamma -1}}{(K \lambda)^{\gamma-p}} \iint_{2^4 Q_i } |Dv_i|^\gamma \, \d x \d t.
\end{align} 
Since~\eqref{eq:v_i-subintinsic} holds, we can choose ${\sf K}={\sf c}^{\frac{1}{p}}$ and apply the higher integrability lemma (Lemma~\ref{lem:HI}) to obtain 
\begin{equation}\label{ineq:application-HI}
\biint_{2^4 Q_i} |Dv_i|^\gamma\d x \d t
\leq c\lambda^\gamma,
\end{equation}
where $c=c(n,N,p, L,\nu)$, which implies
$$
\frac{2^{\gamma -1}}{(K \lambda)^{\gamma-p}} \iint_{2^4 Q_i } |Dv_i|^\gamma \, \d x \d t \leq c\frac{2^{\gamma -1}}{K^{\gamma-p}} \lambda^p |Q_i|.
$$
Using the inequality above together with comparison estimate~\eqref{eq:comparison1} in~\eqref{eq:level-set-first}, we get
\begin{align*}
\iint_{2^3 Q_i \cap \{|Du|> K \lambda\}}& |Du|^p \, \d x \d t \\
&\leq c  \Bigg[ \left( \delta + \frac{\omega(R)^{p'}}{\delta^\frac{(2-p)_+}{p-1}}\right) \iint_{2^5 Q_i} (\mu^2 + |Du|^2 )^{\frac{p}{2}}\, \d x \d t \nonumber\\
& \qquad \qquad +  \frac{c}{\delta^\frac{(2-p)_+ (n+2)}{p(n+2) -n-p}} \left[ \iint_{2^5Q_i} |\sff|^{q'} \, \d x \d t \right]^{\frac{1}{q'} \frac{p(n+2)}{p(n+2) -n-p}}\Bigg] \\
&\qquad \qquad + c\frac{2^{\gamma -1}}{K^{\gamma-p}} \lambda^p |Q_i|\nonumber 
\end{align*}
As a consequence the cylinder $2^5 Q_i$ is subintrinsic,  that is, ~\eqref{eq:stopping-time-above-maxradius} holds, we have
\begin{align*}
   &\iint_{2^5 Q_i} (\mu^2 + |Du|^2)^{\frac{p}{2}} \, \d x \d t \nonumber \le c\big|2^5 Q_i\big| \biint_{2^5 Q_i} (\lambda^p + |Du|^p) \, \d x \d t \le  c\big|2^5 Q_i\big| \lambda^p.  \nonumber\\   
\end{align*}
and
\begin{align}\label{eq:f-estimate}
    &\left[
\iint_{2^5 Q_i}
|\sff|^{q'}
\,\dx\dt
\right]^{\frac1{q'}\frac{p(n+2)}{p(n+2)-n-p}} \nonumber\\
&\qquad=
\frac{\big|2^5 Q_i\big|}{{\sf M}^p}
{\sf M}^p
\big|2^5 Q_i\big|^{\frac{p}{p(n+2)-n-p}}
 \left[
\biint_{2^5 Q_i}
|\sff|^{q'}
\,\dx\dt
\right]^{\frac1{q'}\frac{p(n+2)}{p(n+2)-n-p}}\\
&\qquad \le
\frac{1}{{\sf M}^p}\big|2^5 Q_i\big|\lambda^p\nonumber.
\end{align}
Combining the inequalities above, we obtain
\begin{align} \label{eq:gradu-super}
&\iint_{2^3 Q_i \cap \{|Du|> K \lambda\}} |Du|^p \, \d x \d t \nonumber \\
&\quad \leq c \left( \delta + \frac{\omega(R)^{p'}}{\delta^\frac{(2-p)_+}{p-1}} + \frac{1}{{\sf M}^p \delta^\frac{(2-p)_+(n+2)}{p(n+2) - n -p}} + \frac{2^{\gamma }}{K^{\gamma-p}}\right) \lambda^p |Q_i|.
\end{align}
 As a consequence of ~\eqref{eq:intr}, we have that $Q_i$ is intrinsic, and, together with ~\eqref{exponent_average}, we obtain
\begin{equation} \label{eq:Qi-lambda-eq}
|Q_i| \lambda^p =  \iint_{Q_i} |D u|^p \,\dx\dt
    + 
    {\sf M}^p \bigg[\iint_{Q_i} |\sff|^{q'} \,\dx \dt\bigg]^{\frac{1}{q'}\frac{p(n+2)}{p(n+2)-(n+p)}} . 
\end{equation}
Let $\vartheta>0$ be a parameter to be chosen later, and define $\eta>1$ by
\begin{equation*}
\eta:=\frac{p(n+2)-n}{p},
\end{equation*}
so that
\begin{equation}\label{eq-eta'}
\frac{1}{q'\eta'}\frac{p(n+2)}{p(n+2)-(n+p)}=1, 
\end{equation}
where $q'$ is defined in ~\eqref{eq:q-q'}. By H\"older's inequality together with ~\eqref{eq-eta'}, we obtain
\begin{align*}
    \bigg[
    \iint_{Q_i} &|\sff|^{q'}\,\dx\dt
    \bigg]^{\frac1{q'}\frac{p(n+2)}{p(n+2)-(n+p)}}\\
    &=
    \bigg[
    \iint_{Q_i} |\sff|^{\vartheta}|\sff|^{q'-\vartheta}\,\dx\dt
    \bigg]^{\frac1{q'}\frac{p(n+2)}{p(n+2)-(n+p)}}\\
    &\le
    \bigg[
    \iint_{Q_i} |\sff|^{\vartheta	\eta}\,\dx\dt
    \bigg]^{\frac1{ q' 	\eta }\frac{p(n+2)}{p(n+2)-(n+p)}}\bigg[
    \iint_{Q_i} |\sff|^{(q'-\vartheta)	\eta'}\,\dx\dt
    \bigg]^{\frac{1}{q'	\eta'}\frac{p(n+2)}{p(n+2)-(n+p)}}\\
    &=
    \bigg[
    \iint_{Q_i} |\sff|^{\vartheta\eta}\,\dx\dt
    \bigg]^{\frac1{q' \eta }\frac{p(n+2)}{p(n+2)-(n+p)}}
    \iint_{Q_i} |\sff|^{(q'-\vartheta)\eta'}\,\dx\dt.
\end{align*}
Now, we denote
\begin{equation} \label{def:alpha}
 \alpha = \frac{n+2}{n(p-1) +p + s}.   
\end{equation}
In order to apply Fubini type argument later in the proof, we choose $\vartheta$ such that
\[
\vartheta \eta = \frac{s}{p}  (q'-\vartheta) \eta'
\]
is satisfied, which implies
$$
\vartheta \eta = \alpha s, \quad (q'-\vartheta)\eta' = \alpha p \quad \text{ and } \quad \tfrac1{q' \eta }\tfrac{p(n+2)}{p(n+2)-(n+p)} = \tfrac{p}{p(n+2) - (n+p)}.
$$
With such a choice of $\vartheta$, the preceding Hölder's inequality can be written as
$$
 \bigg[
    \iint_{Q_i} |\sff|^{q'}\,\dx\dt
    \bigg]^{\frac1{q'}\frac{p(n+2)}{p(n+2)-(n+p)}} \leq \bigg[
    \iint_{Q_i} |\sff|^{\alpha s}\,\dx\dt
    \bigg]^{\frac{p}{p(n+2)-(n+p)}}
    \iint_{Q_i} |\sff|^{\alpha p}\,\dx\dt.
$$
Furthermore, ~\eqref{eq:Qi-lambda-eq} can be estimated as
\begin{equation} \label{eq:Qi-lambda-ineq}
    |Q_i| \lambda^p \leq \iint_{Q_i} |Du|^p \, \d x\d t +  {\sf M}^p \sfF^p \iint_{Q_i} |\sff|^{\alpha p} \, \d x \d t,
\end{equation}
where we denoted
$$
\sfF = \left[ \iint_{Q_R} |\sff|^{\alpha s} \, \d x \d t \right]^\frac{1}{p(n+2) - (n+p)}.
$$
 We decompose the integration domains into superlevel and sublevel regions to estimate the right-hand side of~\eqref{eq:Qi-lambda-ineq}. Hence,
\begin{align*}
    &\iint_{Q_i} |Du|^p \, \d x\d t +  {\sf M}^p \sfF^p \iint_{Q_i} |\sff|^{\alpha p} \, \d x \d t  \\
    &\quad \leq \iint_{Q_i \cap \{|Du| > \lambda/4\}} |Du|^p \, \d x\d t +  {\sf M}^p \sfF^p \iint_{Q_i\cap \{ |\sff|^\alpha > \lambda/(4 {\sf M} \sfF)\}} |\sff|^{\alpha p} \, \d x \d t + \tfrac{1}{2} |Q_i| \lambda^p,
\end{align*}
which yields  
\begin{align*}
|Q_i| \lambda^p \leq 2 \iint_{Q_i \cap \{ |Du| > \lambda /4 \}} |Du|^p  \, \d x\d t + 2 {\sf M}^p \sfF^p \iint_{Q_i \cap \{ |\sff|^\alpha > \lambda/(4 {\sf M} \sfF)\}} |\sff|^{\alpha p } \, \d x \d t.
\end{align*}
Combining this with~\eqref{eq:gradu-super}, we have obtained
\begin{align*} 
&\iint_{2^3 Q_i \cap \{|Du|> K \lambda\}} |Du|^p \, \d x \d t \nonumber \\
&\quad \leq c {\sf Q} \left[ \iint_{Q_i \cap \{ |Du| > \lambda /4 \}} |Du|^p  \, \d x\d t + {\sf M}^p \sfF^p \iint_{Q_i \cap \{ |\sff|^\alpha > \lambda/(4 {\sf M} \sfF)\}} |\sff|^{\alpha p } \, \d x \d t \right],
\end{align*}
where we denoted
\[
{\sf Q} =  \delta + \frac{\omega(R)^{p'}}{\delta^\frac{(2-p)_+}{p-1}} + \frac{1}{{\sf M}^p \delta^\frac{(2-p)_+(n+2)}{p(n+2) - n -p}} + \frac{2^{\gamma }}{K^{\gamma-p}}.
\]

\subsubsection{On superlevel estimates: from local to global scales}
Recall that the collection $\{2^3 Q_i\}_{i \in \N}$ covers the superlevel set $\sfE(\lambda,r_1)$ defined in ~\eqref{def-superlevel set}. Since $K\geq1$, we have $\sfE(K\lambda,r_1)\subset\sfE(\lambda,r_1)$, and hence the same family also covers $\sfE(K\lambda,r_1)$.
Furthermore, since the cylinders $Q_i$, with $i \in \mathbb{N}$, are pairwise disjoint and contained in $Q_{r_2}$, the preceding estimate implies
\begin{align*}
    &\iint_{\sfE(K\lambda,r_1)} |Du|^p \, \d x \d t \\
    &\quad \leq \sum_{i=1}^\infty \iint_{2^3 Q_i \cap \{|Du|> K \lambda\}} |Du|^p \, \d x \d t \\
    &\quad \leq c {\sf Q} \sum_{i=1}^\infty \left[ \iint_{Q_i \cap \{ |Du| > \lambda /4 \}} |Du|^p  \, \d x\d t + {\sf M}^p \sfF^p \iint_{Q_i \cap \{ |\sff|^\alpha > \lambda/(4 {\sf M} \sfF)\}} |\sff|^{\alpha p } \, \d x \d t \right] \\
    &\quad \leq c {\sf Q} \left[ \iint_{\sfE (\lambda/4, r_2)} |Du|^p  \, \d x\d t + {\sf M}^p \sfF^p \iint_{Q_{r_2} \cap \{ |\sff|^\alpha > \lambda/(4 {\sf M} \sfF)\}} |\sff|^{\alpha p } \, \d x \d t \right].
\end{align*}

To proceed with the analysis, we introduce \emph{truncated superlevel sets}. More precisely, for $k\geq {\sf B}\lambda_o$, we define
\begin{equation}\label{def:truncated-sets}
\sfE_k(\lambda,r_1)
:=
\Big\{
z\in Q_{r_1} :
z \text{ is a Lebesgue point of } |Du|
\text{ and }
|Du|_k(z)>\lambda
\Big\},
\end{equation}
where
\[
|Du|_k:=\min\{|Du|,k\}
\]
stands for the truncation of $|Du|$ at level $k$. Using this notation, the preceding estimate can be expressed as 
\begin{align} \label{eq:truncated-levelset-est}
&\iint_{\sfE_k(K\lambda,r_1)} |Du|^p \, \d x \d t \nonumber \\
&\quad \leq  c {\sf Q} \left[ \iint_{\sfE_k (\lambda/4, r_2)} |Du|^p  \, \d x\d t + {\sf M}^p \sfF^p \iint_{Q_{r_2} \cap \{ |\sff|^\alpha > \lambda/(4 {\sf M} \sfF)\}} |\sff|^{\alpha p } \, \d x \d t \right],
\end{align}
since $\sfE_k (K \lambda,r_1) = \varnothing$ if $k \leq K \lambda$, and $\sfE_k (K \lambda,r_1) = \sfE (K \lambda,r_1)$ and $\sfE_k (\lambda/4,r_2) = \sfE (\lambda/4,r_2)$ if $k > K \lambda$. 

 Next, we recall that $s \in (p,p+\tfrac{n}{4} + \varepsilon)$ and $\gamma = \tfrac12 (s + p + \tfrac{4}{n} + \varepsilon)$ and multiply the inequality~\eqref{eq:truncated-levelset-est} by $\lambda^{s-p-1}$, and integrate it with respect to $\lambda$ over interval $({\sf B} \lambda_o, \infty)$. On the left hand side, by Fubini's theorem we have
\begin{align*}
    &\int_{{\sf B} \lambda_o}^\infty \lambda^{s-p-1} \iint_{\sfE_k(K\lambda,r_1)} |Du|^p \, \d x \d t \, \d \lambda \\
    &\quad = \iint_{Q_{r_1}}  |Du|^p \int_{{\sf B} \lambda_o}^{|Du|_k/K} \lambda^{s-p-1} \, \d \lambda \, \d x \d t \\
    &\quad = \frac{1}{s-p} \left[ K^{- (s-p)} \iint_{Q_{r_1}} |Du|_k^{s-p} |Du|^p\, \d x \d t - ({\sf B} \lambda_o)^{s-p} \iint_{Q_{r_1}} |Du|^p \, \d x \d t \right].
\end{align*}
Similarly, we obtain
\begin{align*}
   \int_{{\sf B} \lambda_o}^\infty \lambda^{s-p-1} \iint_{\sfE_k (\lambda/4, r_2)} |Du|^p  \, \d x\d t \, \d \lambda \leq \frac{4^{s-p}}{s-p}  \iint_{Q_{r_2}} |Du|_k^{s-p} |Du|^p\, \d x \d t,
\end{align*}
and 
\begin{align*}
    &\int_{{\sf B} \lambda_o}^\infty \lambda^{s-p-1} {\sf M}^p \sfF^p \iint_{Q_{r_2} \cap \{ |\sff|^\alpha > \lambda/(4 {\sf M} \sfF)\}} |\sff|^{\alpha p } \, \d x \d t   \, \d \lambda \\
    &\quad \leq \frac{4^{s-p}}{s-p} {\sf M}^s \sfF^s \iint_{Q_{r_2}} |\sff|^{\alpha s } \, \d x \d t.
\end{align*}
Combining the estimates above,  we obtain
\begin{align} \label{eq:trunc-est}
&\iint_{Q_{r_1}} |Du|_k^{s-p} |Du|^p \, \d x \d t \nonumber \\
&\quad \leq c_* K^{s-p} {\sf Q} \iint_{Q_{r_2}} |Du|_k^{s-p} |Du|^p \, \d x \d t \\
&\qquad + (K {\sf B} \lambda_o)^{s-p} \iint_{Q_{r_1}}|Du|^p \, \d x \d t + c_* K^{s-p} {\sf Q} {\sf M}^s \sfF^s \iint_{Q_{r_2}} |\sff|^{\alpha s} \, \d x \d t, \nonumber
\end{align} 
in which  $c_* = c_* (s,p,n,N,\nu,L)$, and
$${\sf Q} = \delta + \frac{\omega(R)^{p'}}{\delta^\frac{(2-p)_+}{p-1}} + \frac{1}{{\sf M}^p \delta^\frac{(2-p)_+(n+2)}{p(n+2) - n -p}} + \frac{2^{\gamma }}{K^{\gamma-p}}.
$$
We choose $K$ large enough so that 
$$
c_* K^{s-p} 2^{\gamma} K^{p - \gamma} = \tfrac18 \iff K = \left( 8 c_* 2^{\gamma} \right)^\frac{1}{\gamma -s}.
$$
Similarly, we choose $\delta$ such that
$$
c_* K^{s-p} \delta = \tfrac18 \iff \delta = \tfrac{1}{8 c_*} (8c_* 2^{\gamma})^\frac{p-s}{\gamma-s},
$$
and afterwards,  we choose an upper bound $R_o>0$ small enough for the radius $R$ such that 
$$
c_* K^{s-p} \delta^{- \frac{(2-p)_+}{p-1}} \omega(R_o)^{p'} = \tfrac{1}{8},
$$
and finally $\sf M$ such that
$$
c_* K^{s-p} {\sf M}^{-p} \delta^{- \frac{(2-p)_+(n+2)}{p(n+2) - n-p}} = \tfrac18, $$
that is, 
\begin{align*}
{\sf M} &= (8 c_*)^{\frac1p\left[1 + \frac{(2-p)_+(n+2)}{p(n+2) - n-p}\right] } (8c_*2^{\gamma})^{\frac{s-p}{(\gamma - s)p}\left[1 + \frac{(2-p)_+(n+2)}{p(n+2) - n-p}\right]} \\
&= \left[ (8c_*)^{\gamma-p} 2^{\gamma (s-p)} \right]^{\frac{1}{p (\gamma-s)} \left[ 1 + \frac{(2-p)_+(n+2)}{p(n+2) - n-p}\right]}
\end{align*}

Observe that $K, \delta$ and $\sf M$ depend on $s,p,n,N,\nu$ and $L$, while $R_o$ depends on the same parameters and additionally on $\omega (\cdot)$. These choices imply $ c_* K^{s-p} {\sf Q} \leq \tfrac12$. Then, by using the definition of $\sf B$ in~\eqref{def:B} and  by applying the iteration lemma~\cite[Lemma 6.1]{Giusti}, and subsequently passing to the limit as $k \to \infty$ together with application of Fatou's lemma, we have
\begin{align*}
    &\iint_{Q_{\frac12 R}} |D u|^s  \, \dx\dt  \\
    &\quad\leq c \lambda_o^{s-p} \iint_{Q_R} |Du|^p \, \d x \d t + c \left[ \iint_{Q_R} |\sff|^{\alpha s} \, \d x \d t \right]^{\frac{s}{p(n+2)- (n+p)} +1}.
\end{align*}
Using the definition of $\alpha$ in~\eqref{def:alpha}, and after taking mean values on both sides and raising the resulting inequality to the power $\frac{1}{s}$, we obtain
\begin{align*}
    \bigg[\biint_{Q_{\frac12 R}} &|D u|^s  \, \dx\dt\bigg]^{\frac1s} \\
   &\leq 
   c\lambda_o^{1-\frac{p}{s}}
   \bigg[ \biint_{Q_{R}} |D u|^p \, \dx\dt\bigg]^{\frac{1}{s}}
    + 
   c 
   \bigg[ \biint_{Q_{R}} |R\sff|^{\alpha s} \, \dx\dt\bigg]^{\frac1{\alpha s} \frac{n+2}{n(p-1)+p}} \\
    &\leq 
    c\lambda_o + 
   c
   \bigg[ \biint_{Q_{R}} |R\sff|^{\alpha s} \, \dx\dt\bigg]^{\frac1{\alpha s} \frac{n+2}{n(p-1)+p}},
\end{align*}
where $\lambda_o\ge 1$ and $d\ge 1$. It follows from the definition of $\lambda_o$  in~\eqref{def:lambda_o} together with $R\le R_o \le 1 $ that 
\begin{align*}
    \lambda_o
    \le  
    c\Bigg[
    \biint_{Q_R}|Du|^p\,\dx\dt
    +
    \bigg[
    \biint_{Q_R}|R\sff|^{\alpha s}\,\dx\dt + 1
    \bigg]^{\frac{1}{\alpha s}\frac{p(n+2)}{p(n+2)-(n+p)}} \Bigg]^\frac{d}{p}.
\end{align*}
Indeed, by applying H\"older inequality, using $q'$ defined in~\eqref{eq:q-q'} and $\alpha s \le q'$, we obtain
\[
 \bigg[
    \biint_{Q_R}|R\sff|^{q'}\,\dx\dt
    \bigg]^{\frac{1}{q'}\frac{p(n+2)}{p(n+2)-(n+p)}}  
    \le  \bigg[
    \biint_{Q_R}|R\sff|^{\alpha s}\,\dx\dt
    \bigg]^{\frac{1}{\alpha s}\frac{p(n+2)}{p(n+2)-(n+p)}}.
\]
Finally, we combine the inequalities above to obtain
\begin{align} \label{eq:final-est}
    \bigg[\biint_{Q_{\frac12 R}} &|D u|^s  \, \dx\dt\bigg]^{\frac1s} \nonumber \\
    &\leq 
    c \Bigg[
    \biint_{Q_R}|Du|^p\,\dx\dt
    +
    \bigg[
    \biint_{Q_R}|R\sff|^{\alpha s}\,\dx\dt + 1
    \bigg]^{\frac{1}{\alpha s}\frac{p(n+2)}{p(n+2)-(n+p)}} \Bigg]^\frac{d}{p}
\end{align}
which holds, whenever $0<R \leq R_o$. In addition, if $R_o < R \leq 1$, by standard covering argument we cover the cylinder $Q_{R/2}$ by a collection of cylinders $\{Q_{R_o/4}(z_j)\}_{j=1}^l$, where $z_j \in Q_{R/2}$ for every $j$, and 
\[
l \leq c(n) \left( \frac{R}{R_o}\right)^{n+2} \leq c(n) R_o^{-n-2},
\]
since $R \leq 1$. Observe that $Q_{R_o/2}(z_j) \subset Q_R$ for every $j$. Therefore, the desired estimate holds for any $R \in (0, 1]$, with the constant $c$ depending additionally on $\omega (\cdot)$ through $R_o$.

\subsection{Proof of Theorem~\ref{thm:main} with VMO coefficients}

Suppose that $\sfA (x,t,\xi) = e(x) \sfa (t,\xi)$, where $\sfa$ satisfies~\eqref{bounds-A} and $e$ satisfies~\eqref{eq:e-uniform-ell} and~\eqref{eq:e-VMO}. Let $v_i$ be a weak solution in the function space~\eqref{eq:vi-space} to
\begin{equation} \label{sol:C-D_coefficients}
\left\{
\begin{array}{cl}
      \partial_t v_i - \Div ( e(x) \sfa (t,Dv_i) )
      = 0,
      & 
      \mbox{in $ 2^5 Q_i$,} \\[7pt]
      v_i = u, & \mbox{on  $\partial_{\rm par} \big(2^5 Q_i\big)$,}
   \end{array}
\right.
\end{equation}
where $u$ is a weak solution to~\eqref{par-sys}.

We apply Lemma~\ref{lem:comparison-vmo}, we obtain
\begin{align} \label{eq:comparison-vmo-1}
\iint_{2^5 Q_i} |Du- Dv_i|^p \, \d x\d t &\leq \delta \iint_{2^5 Q_i} (\mu^2 + |Du|^2)^\frac{p}{2} \, \d x \d t \\
&\quad + \frac{c}{\delta^\frac{(2-p)_+ (n+2)}{p(n+2) -n-p}}\left[ \iint_{2^5 Q_i} |\sff|^{q'} \, \d x \d t \right]^{\frac{1}{q'} \frac{p(n+2) -n}{p(n+2) -n-p}}, \nonumber
\end{align}
for any $\delta \in (0,1]$ and $c = c(n,N,p,\nu,L)$. We now argue as in~\eqref{eq:v_i-subintinsic}. Setting $\delta=1$ in the inequality above and using the facts that $\mu\leq1\leq\lambda$ and $\sf M>1$, we obtain
\begin{equation} \label{eq:vi-subintrinsic}
    \biint_{2^5 Q_i} |Dv_i|^p \, \d x \d t \leq 2^p\biint_{2^5 Q_i} |Du-Dv_i|^p \, \d x \d t + 2^p\biint_{2^5 Q_i} |Du|^p \, \d x \d t \leq c \lambda^p,
\end{equation}
for $c = c(n,N,p,\nu,L)$.

Next, let us consider
\begin{align*}
    w_i&\in C
    \big(
    [t_i-\lambda^{2-p}(2^{4}\rho_i)^2,
    t_i+\lambda^{2-p}(2^{4}\rho_i)^2);
    L^2 \big(B_{2^{4}\rho_i}(x_i);\R^k
    \big)\big)\\
    &\phantom{\in\,}
    \cap
    L^p\big(
    t_i-\lambda^{2-p}(2^{4}\rho_i)^2,
    t_i+\lambda^{2-p}(2^{4}\rho_i)^2;
    W^{1,p} \big(B_{2^{4}\rho_i}(x_i);\R^k
    \big)\big)
\end{align*}
 be the unique weak solution to
\begin{equation*} 
\left\{
\begin{array}{cl}
      \partial_t w_i - \Div (e_i \sfa (t,Dw_i) )
      = 0,
      & 
      \mbox{in $ 2^4 Q_i$,} \\[7pt]
      w_i = v_i, & \mbox{on  $\partial_{\rm par} \big(2^4 Q_i\big)$,}
   \end{array}
\right.
\end{equation*}
in which $e_i = \displaystyle \bint_{2^4 B_i} e(x) \, \d x$ and $v_i$ is a weak solution in~\eqref{sol:C-D_coefficients}. Subtracting the equations satisfied by $v_i$ and $w_i$ yields
\begin{align*}
&\partial_t (v_i - w_i) - \Div ( e_i (\sfa (t,Dv_i) - \sfa(t,Dw_i) ) ) = \Div ( (e(x) - e_i) \sfa(t,Dv_i)) \quad \text{ in } 2^4 Q_i,
\end{align*}
with $w_i-v_i = 0$ on $\partial_{\rm par} \left( 2^4 Q_i\right)$.  Arguing as in the proof of Lemma~\ref{lem:comparison} using Steklov averages (that is, testing formally with $v_i -w_i$), we obtain
\begin{align*}
    &\iint_{2^4 Q_i} e_i (\sfa(t,Dv_i) - \sfa(t,Dw_i) ) \cdot (Dv_i - Dw_i)  \, \d x \d t \\
    &\quad \leq \iint_{2^4 Q_i} |e(x) - e_i| |\sfa(t,Dv_i)| |Dv_i - Dw_i| \, \d x \d t.
\end{align*}
 Using the growth conditions to $\sfa$ in~\eqref{bounds-A} and bounds for $e$ and $e_i$ in~\eqref{eq:e-uniform-ell} combined with the Young's inequality with exponents $p$ and $p'$, we obtain
\begin{align} \label{eq:vmo-est}
&C \iint_{2^4 Q_i} (\mu^2 + |Dv_i|^2 + |Dw_i|^2)^\frac{p-2}{2} |Dv_i - Dw_i|^2 \, \d x \d t\nonumber \\
&\quad \leq \iint_{2^4 Q_i} |e(x) - e_i| |\sfa (t, Dv_i)| |Dv_i - Dw_i| \, \d x \d t \nonumber \\
&\quad \leq L \iint_{2^4 Q_i} |e(x) - e_i|(\mu^2 + |Dv_i|^2)^\frac{p-1}{2} |Dv_i - Dw_i| \, \d x \d t \nonumber \\
&\quad \leq c \iint_{2^4 Q_i} (\mu^2 + |Dv_i|^2)^\frac{p-1}{2} |Dv_i - Dw_i| \, \d x \d t  \\
&\quad \leq c \iint_{2^4 Q_i} (\mu^2 + |Dv_i|^2)^\frac{p}{2} +  (\mu^2 + |Dv_i|^2)^\frac{p-1}{2}|Dw_i| \, \d x \d t \nonumber \\
&\quad \leq c \kappa^{-\frac{1}{p-1}} \iint_{2^4 Q_i} (\mu^2 + |Dv_i|^2)^\frac{p}{2} \, \d x \d t +  \kappa \iint_{2^4 Q_i} |Dw_i|^p \, \d x\d t \nonumber
\end{align}
for any $\kappa \in (0,1]$. Observe that, by Lemma~\ref{lem:tech} together with the estimate in~\eqref{eq:vmo-est}
\begin{align*}
\iint_{2^4 Q_i} |Dw_i|^p \, \d x \d t &\leq \tilde c \iint_{2^4 Q_i} (\mu^2 + |Dv_i|^2)^\frac{p}{2} \, \d x \d t \\
&\quad + \tilde c \iint_{2^4 Q_i} (\mu^2 + |Dv_i|^2 + |Dw_i|^2)^\frac{p-2}{2} |Dv_i - Dw_i|^2 \, \d x \d t\\
&\leq c_{\kappa}\iint_{2^4 Q_i} (\mu^2 + |Dv_i|^2)^\frac{p}{2} \, \d x \d t + \frac{\tilde c \kappa}{C} \iint_{2^4 Q_i} |Dw_i|^p \, \d x\d t.
\end{align*}
By choosing $\kappa = \frac{C}{2 \tilde c }$ in the inequality above, using the subintrinsic property of $v_i$ in ~\eqref{eq:vi-subintrinsic} and $\mu \leq 1 \leq \lambda$ implies
\begin{equation} \label{eq:w-i-subintrinsic}
\biint_{2^4 Q_i} |Dw_i|^p \, \d x \d t \leq c \biint_{2^4 Q_i} (\mu^2 + |Dv_i|^2)^\frac{p}{2}\, \d x \d t  \leq c \lambda^p,
\end{equation}
for $c = c(n,N,p,\nu,L)$. Moreover, we can estimate~\eqref{eq:vmo-est} further as
\begin{align} \label{eq:comparison-v-w}
&C \iint_{2^4 Q_i} (\mu^2 + |Dv_i|^2 + |Dw_i|^2)^\frac{p-2}{2} |Dv_i - Dw_i|^2 \, \d x \d t \nonumber\\
&\quad \leq L \iint_{2^4 Q_i} |e(x) - e_i|(\mu^2 + |Dv_i|^2)^\frac{p-1}{2} |Dv_i - Dw_i| \, \d x \d t \nonumber \\
&\quad \leq \tilde \delta^{- \frac{1}{p-1}} L^{p'} \iint_{2^4 Q_i} |e(x)-e_i|^\frac{p}{p-1}  (\mu^2 + |Dv_i|^2)^\frac{p}{2}\, \d x \d t \\
&\qquad + \tilde \delta \iint_{2^4 Q_i} |Dv_i - Dw_i|^p \, \d x \d t \nonumber
\end{align}
for any $\tilde \delta > 0$. Let  $\sigma = \sigma (p,n,N,L,\nu)> p$ denote the parameter from Lemma~\ref{lem:HI-KL}. By using H\"older's inequality with exponents $\frac{\sigma}{p}$ and $\frac{\sigma}{\sigma-p}$ together with Lemma~\ref{lem:HI-KL} for $v_i$ (which is applicable due to~\eqref{eq:vi-subintrinsic}),~\eqref{eq:e-uniform-ell}, and the fact that $\mu \leq 1 \leq \lambda$, we obtain 
\begin{align*}
&\iint_{2^4 Q_i} |e(x)-e_i|^\frac{p}{p-1}  (\mu^2 + |Dv_i|^2)^\frac{p}{2}\, \d x \d t \\
&\quad \leq |2^{4} Q_i| \left[ \biint_{2^4 Q_i} |e(x) - e_i|^{\frac{p}{p-1} \frac{\sigma}{\sigma-p}}\, \d x \d t  \right]^\frac{\sigma-p}{\sigma} \left[ \biint_{2^4 Q_i} (\mu^2 + |Dv_i|^2)^\frac{\sigma}{2} \, \d x \d t \right]^\frac{p}{\sigma} \\
&\quad \leq c |2^{4} Q_i| (2L)^{\frac{p}{p-1} - \frac{\sigma-p}{\sigma}} \left[ \bint_{2^4 B_i} |e(x) - e_i| \, \d x  \right]^\frac{\sigma-p}{\sigma} \left[ \biint_{2^4 Q_i} (\mu^2 + |Dv_i|^2)^\frac{\sigma}{2} \, \d x \d t \right]^\frac{p}{\sigma} \\
&\quad \leq c |2^{4} Q_i| \omega(R)^\frac{\sigma-p}{\sigma} \lambda^p.
\end{align*}
Combining ~\eqref{eq:comparison-v-w} with the estimate above, we get
\begin{align}\label{ineq:comparion-v-w-2}
& C \iint_{2^4 Q_i} (\mu^2 + |Dv_i|^2 + |Dw_i|^2)^\frac{p-2}{2} |Dv_i - Dw_i|^2 \, \d x \d t \nonumber\\
&\le \tilde \delta \iint_{2^4 Q_i} |Dv_i - Dw_i|^p \, \d x \d t + c\tilde \delta^{- \frac{1}{p-1}} |2^{4} Q_i| \omega(R)^\frac{\sigma-p}{\sigma} \lambda^p.   
\end{align} 
Recall that by Lemma~\ref{lem:tech}
$$
(\mu^2 + |Dv_i|^2 + |Dw_i|^2)^\frac{p-2}{2} |Dv_i - Dw_i|^2 \geq \bar c|Dv_i - Dw_i|^p\quad \text{ if } p \geq 2,
$$
and~\eqref{eq:p-difference-sing} holds true in the case $p<2$. Using these facts in~\eqref{ineq:comparion-v-w-2} together with the choices $\tilde \delta = \tfrac{C \bar c}{4}$ in the case $p \geq 2$ and $\tilde \delta = \tfrac{C }{4} \delta^\frac{2-p}{p}$ in the case $p < 2$ yield
\begin{align} \label{eq:comparison-vmo-2}
\iint_{2^4 Q_i} |Dv_i - Dw_i|^p \, \d x \d t &\leq \delta \iint_{2^4 Q_i} (\mu^2 + |Dv_i|^2)^\frac{p}{2} \, \d x \d t + c \frac{\omega(R)^\frac{\sigma-p}{\sigma}}{\delta^\frac{(2-p)_+}{p(p-1)}}  |2^4 Q_i| \lambda^p \nonumber \\
&\leq c \left( \delta +  \frac{\omega(R)^\frac{\sigma-p}{\sigma}}{\delta^\frac{(2-p)_+}{p(p-1)}} \right)|2^4 Q_i| \lambda^p,
\end{align}
for any $\delta \in (0,1]$, where $c = c(n,N,p,\nu,L)$.

\subsubsection{Estimates on the superlevel sets of $|Du|$ and conclusion}
Observe that due to~\eqref{eq:w-i-subintrinsic}, $w_i$ satisfies the higher integrability from Lemma~\ref{lem:HI}.
Let $\varepsilon = \varepsilon (n,N,p,\nu,L) > 0$ be the constant from the aforementioned lemma, let $s \in (p,p + \tfrac{4}{n} + \varepsilon)$ and $\gamma = \tfrac12 ( s+ p + \tfrac{4}{n} +\varepsilon )$.
By applying Lemma~\ref{lem:elementary-superlevel} with parameters $(a, b, \alpha, \beta, K, k) = (Du, Dw_i, p, \gamma, K\lambda,Nn)$, we have
\begin{align*}
&\iint_{2^3 Q_i \cap \{|Du|> K \lambda\}} |Du|^p \, \d x \d t \\
&\quad \leq 2^p \iint_{2^4 Q_i } |Du - D w_i|^p \, \d x \d t + \frac{2^{\gamma -1}}{(K \lambda)^{\gamma-p}} \iint_{2^4 Q_i } |Dw_i|^\gamma \, \d x \d t \\
&\quad \leq 4^{p} \iint_{2^4 Q_i } |Du - D v_i|^p \, \d x \d t +4^{p} \iint_{2^4 Q_i } |Dv_i - D w_i|^p \, \d x \d t \\
&\qquad + \frac{2^{\gamma -1}}{(K \lambda)^{\gamma-p}} \iint_{2^4 Q_i } |Dw_i|^\gamma \, \d x \d t.
\end{align*} 
By using~\eqref{eq:comparison-vmo-1},~\eqref{eq:comparison-vmo-2}, and Lemma~\ref{lem:HI} for $w_i$ (which is applicable due to~\eqref{eq:w-i-subintrinsic}), as in Section~\ref{sec:proof-cont-A} we obtain
\begin{align*}
&\iint_{2^3 Q_i \cap \{|Du|> K \lambda\}} |Du|^p \, \d x \d t \\
&\quad \leq \delta \iint_{2^5 Q_i} (\mu^2 + |Du|^2)^\frac{p}{2} \, \d x \d t + \frac{c}{\delta^\frac{(2-p)_+ (n+2)}{p(n+2) -n-p}}\left[ \iint_{2^5 Q_i} |\sff|^{q'} \, \d x \d t \right]^{\frac{1}{q'} \frac{p(n+2) -n}{p(n+2) -n-p}} \nonumber \\
&\quad \quad +c \left( \delta +  \frac{\omega(R)^\frac{\sigma-p}{\sigma}}{\delta^\frac{(2-p)_+}{p(p-1)}} \right)|2^4 Q_i| \lambda^p \\
&\qquad + \frac{c 2^{\gamma -1}}{K^{\gamma-p}} \lambda^{p}|Q_i|.
\end{align*} 
Finally, using once more the subintrinsicity of the cylinder $2^5 Q_i$, i.e.,~\eqref{eq:stopping-time-above-maxradius}, we infer
\[
\iint_{2^5 Q_i} (\mu^2 + |Du|^2)^\frac{p}{2}\, \d x \d t  \le  c\big|2^5 Q_i\big| \lambda^p.
\]
and combined with the estimate for the right-hand side $\sff$ in~\eqref{eq:f-estimate}, we have
\begin{align*}
&\iint_{2^3 Q_i \cap \{|Du|> K \lambda\}} |Du|^p \, \d x \d t \\
&\quad \leq c \left(\delta + \frac{1}{\delta^\frac{(2-p)_+ (n+2)}{p(n+2)-n-p} {\sf M}^p} + \frac{\omega(R)^\frac{\sigma-p}{\sigma}}{\delta^\frac{(2-p)_+}{p(p-1)}} + \frac{2^\gamma}{K^{\gamma -p }}\right) |Q_i| \lambda^p,
\end{align*}
where $c = c (n,N,p,\nu,L)$. 
At this point, we obtain the corresponding estimate to~\eqref{eq:gradu-super} in the previous proof. The proof follows exactly the same lines as the previous one in Section~\ref{sec:proof-cont-A}, we highlight the main changes. In particular, we replace $\sf Q$ by $\tilde{\sf Q}$, defined by
\[
{\tilde{\sf Q}} = \delta + \frac{1}{\delta^\frac{(2-p)_+ (n+2)}{p(n+2)-n-p} {\sf M}^p} + \frac{\omega(R)^\frac{\sigma-p}{\sigma}}{\delta^\frac{(2-p)_+}{p(p-1)}} + \frac{2^\gamma}{K^{\gamma -p }}.
\]
In order to apply the iteration lemma~\cite[Lemma 6.1]{Giusti} in~\eqref{eq:trunc-est}, this corresponds of choosing the threshold $R_o>0$ sufficiently small so that
\[
c_* K^{s-p}
\delta^{-\frac{(2-p)_+}{p(p-1)}}
\omega(R_o)^{\frac{\sigma -p}{\sigma}}
=
\tfrac{1}{8}.
\]
Consequently, we arrive at the estimate~\eqref{eq:final-est}, which holds for all $0<R\leq R_o$, and by the same covering argument we deduce the result for all $0<R\leq 1$ concluding the proof of Theorem~\ref{thm:main}.

\medskip\medskip
\noindent
{\bf Acknowledgments.} 
 This research was funded in whole or in part by the Austrian Science Fund (FWF) [10.55776/P36295]. K. Moring was partially supported by the
Early Career Grant program 2025 at the University of Salzburg. For open access purposes, the author has applied a CC BY public copyright license to any author accepted manuscript version arising from this submission.


\begin{thebibliography}{99.}%


\bibitem{Acerbi-Min-p(x)}
\newblock E.~Acerbi and G.~Mingione.
\newblock Gradient estimates for the $p(x)$-Laplacean system. 
\newblock\emph{J. Reine Angew. Math.} 584 (2005), 117-148. 


\bibitem{Acerbi-Min}
\newblock E.~Acerbi and G.~Mingione.
\newblock Gradient estimates for a class of parabolic systems.
\newblock \emph{Duke Math. J.} 136 (2007), no. 2, 285-320.

\bibitem{ABDM}
\newblock P.~Andrade, V.~Bögelein, F. Duzaar, and K. Moring. 
\newblock Calderón-Zygmund estimates for parabolic $p$-Laplacian systems with non-divergence form right-hand sides. 
\newblock \emph{Ric. Mat.} (2026).


\bibitem{Bogo}
M.~E.~Bogovski\u{\i}
\newblock
Solutions of some problems of vector analysis, associated with the operators div and grad.
\newblock \emph{Trudy Sem. S. L. Sobolev}, no.~1 (1980), 5-40.


\bibitem{BDGLS}
\newblock V.~B\"ogelein, F.~Duzaar, U.~Gianazza, N.~Liao, and C.~Scheven.
\newblock Schauder estimates for parabolic $p$-Laplace
systems
\newblock\emph{J. Lond. Math. Soc. (2)}, to appear,  DOI:10.1112/jlms.70463


\bibitem{BDLM-1}
\newblock V.~B\"ogelein, F.~Duzaar, N.~Liao and K.~Moring.
\newblock Gradient estimates for the fractional $p$-Poisson equation.
\newblock \emph{Journal de Math\'ematiques Pures et Appliqu\'ees (9)} 204 (2025), Paper No. 103764, 25 pp.

\bibitem{BDLM-2}
\newblock V.~B\"ogelein, F.~Duzaar, N.~Liao and K.~Moring.
\newblock Sharp gradient integrability for $(s,p)$-Poisson type equations.
\newblock preprint, \emph{arXiv:2602.08944}.

\bibitem{BDM-obst}
\newblock V.~B\"ogelein, F.~Duzaar, and G.~Mingione.
\newblock Degenerate problems with irregular obstacles. 
\newblock \emph{J. Reine Angew. Math.} 650 (2011), 107-160.


\bibitem{By-Ki-Ku}
\newblock S.~S.~Byun, K.~Kim, and D.~Kumar.
\newblock Gradient estimates for mixed local and nonlocal parabolic problems with measure data. 
\newblock\emph{J. Math. Anal. Appl.} 538 (2024), no. 2, Paper No. 128351, 30 pp.

\bibitem{Byun-Wang}
\newblock S.~S.~Byun and L.~Wang.
\newblock Nonlinear gradient estimates for elliptic equations of general type. 
\newblock\emph{Calc. Var. Partial Differential Equations} 45 2012, no. 3-4, 403–419.

\bibitem{Caff-Peral}
L.~A.~Caffarelli and I.~Peral.
\newblock On $W^{1,p}$ estimates for elliptic equations in divergence form.
\newblock\emph{Comm. Pure Appl. Math.} 51 (1998), no.~1, 1-21.

\bibitem{Calderon-Zygmund-1952}
\newblock A.~P. ~Calder\'on, A.~Zygmund. 
\newblock On the existence of certain singular integrals. 
\newblock{\emph Acta Math.} 88 (1952), 85–139.

\bibitem{DiBe}
E.~DiBenedetto,
\newblock {\em Degenerate parabolic equations,}
\newblock Springer-Verlag, Universitytext xv, 387, New York, NY, 1993.

\bibitem{DiBenedetto_Holder}
E.~DiBenedetto and A.~Friedman.
\newblock H\"older estimates for nonlinear degenerate parabolic systems.
\newblock {\em J. Reine Angew. Math.}, 357 (1985), 1-22.

\bibitem{DiBenedetto-Manfredi}
E.~DiBenedetto and J.~Manfredi.
\newblock On the higher integrability of the gradient of weak solutions of certain degenerate elliptic systems. \newblock\emph{Amer. J. Math.} 115 (1993), no.~5, 1107-1134.

\bibitem{Du-Mi-pot}
F.~Duzaar and G.~Mingione. 
\newblock Gradient estimates via non-linear potentials. 
\newblock\emph{Amer. J. Math.} 133 (2011), no.~4, 1093-11149. 

\bibitem{Du-Mi-St}
F.~Duzaar, G.~Mingione and K.~Steffen.
\newblock Parabolic systems with polynomial growth and regularity. \newblock\emph{Mem. Amer. Math. Soc.} 214 (2011), no.~1005, x+118 pp.

\bibitem{Galdi}
G.~P.~Galdi. 
An Introduction to the Mathematical Theory of the 
Navier-Stokes Equations.
\newblock Springer Monographs in Mathematics, 2011.

\bibitem{Giusti}
E.~Giusti. 
\newblock\emph{Direct Methods in the Calculus of Variations.} World Scientiﬁc Publishing Company, Tuck Link,
Singapore, 2003.

\bibitem{Iwaniec}
T.~Iwaniec.
\newblock Projections onto gradient fields and $L^p$-estimates for degenerated elliptic operators. 
\newblock{\em Studia Math.} 75 (1983), no.~3, 293-312.



\bibitem{Kinnunen-Lewis:1}
\newblock J.~Kinnunen and J.~L.~Lewis.
\newblock Higher integrability for parabolic systems of $p$-Laplacian-type.
\newblock {\em Duke Math.~J.} 102 (2000), no.~2, 253-271.

\bibitem{Kinnunen-Zhou-1}
J.~Kinnunen and S.~Zhou.
\newblock A local estimate for nonlinear equations with discontinuous coefficients.
\newblock\emph{Comm. Partial Differential Equations} 24 (1999), no.~11-12, 2043-2068.

\bibitem{Kinnunen-Zhou-2}
J.~Kinnunen and S.~Zhou.
\newblock A boundary estimate for nonlinear equations with discontinuous coefficients. 
\newblock \emph{Differential Integral Equations} 14 (2001), no.~4, 475-492. 

\bibitem{Mingione:Measure-data}
G.~Mingione.
\newblock The Calderón-Zygmund theory for elliptic problems with measure data.
\newblock\emph{ Ann. Sc. Norm. Super. Pisa Cl. Sci.} (5) 6 (2007), no.~2, 195–261.

\bibitem{Scheven-1}
C.~Scheven.
\newblock Non-linear Calderón-Zygmund theory for parabolic systems with subquadratic growth. \newblock\emph{J. Evol. Equ.} 10 (2010), no. 3, 597-622. 

\bibitem{Scheven-2}
C.~Scheven.
\newblock Regularity for subquadratic parabolic systems: higher
integrability and dimension estimates. \newblock \emph{Proc. Roy. Soc. Edinburgh Sect. A} 140 (2010), no. 6, 1269--1308.

\bibitem{Sverak-Yan}
 V. \v{S}ver\'ak and X. Yan
\newblock Non-Lipschitz minimizers of smooth uniformly convex functionals. \newblock \emph{Proc. Natl. Acad. Sci} 99 (2002), no. 24, 15269–15276.
\end{thebibliography}
\end{document}